\documentclass[10pt]{article}
\usepackage[utf8]{inputenc}
\usepackage{amsmath}
\usepackage{amsfonts}
\usepackage{amssymb}
\usepackage{amsthm}
\usepackage{subcaption}
\usepackage{graphicx}
\usepackage{fullpage}
\usepackage{mathtools}
\usepackage{enumerate}
\usepackage[margin=1in]{geometry}
\usepackage{pdfpages}
\usepackage{color}
\usepackage{algorithm}
\usepackage{algpseudocode}
\usepackage{booktabs}

\usepackage{url}
\usepackage{hyperref}
\usepackage[mathlines]{lineno}
\usepackage{booktabs}

\usepackage{color}

\newcommand{\Ocal}{\mathcal{O}}
\newcommand{\Gcal}{\mathcal{G}}
\newcommand{\Tcal}{\mathcal{T}}
\newcommand{\Scal}{\mathcal{S}}

\newcommand{\Rbb}{\mathbb{R}}
\newcommand{\Nbb}{\mathbb{N}}

\DeclareMathOperator*{\argmin}{\text{argmin}}

\newtheorem{assumption}{Assumption}
\newtheorem{definition}{Definition}

\newtheorem{lemma}{Lemma}
\newtheorem{theorem}{Theorem}
\newtheorem{remark}{Remark}
\newtheorem{proposition}{Proposition}

\algrenewcommand\algorithmicrequire{\textbf{Input:}}
\algrenewcommand\algorithmicensure{\textbf{Output:}}

\newcommand{\utt}{u_{\theta(t)}}
\newcommand{\ut}{u_{\theta}}

\newcommand{\barOmega}{\bar{\Omega}}
\newcommand{\barTheta}{\bar{\Theta}}
\newcommand{\barvarepsilon}{\bar{\varepsilon}}

\newcommand{\fvt}{f(\cdot,v_t,\nabla v_t)}
\newcommand{\fut}{f(\cdot,u_t,\nabla u_t)}

\definecolor{brickred}{rgb}{0.8, 0.25, 0.33}

\title{Approximation of Solution Operators for High-dimensional PDEs}
\author{Nathan Gaby \and Xiaojing Ye \footnote{Department of Mathematics and Statistics, Georgia State University, Atlanta, GA 30303, USA. Email: \url{{ngaby1,xye}@gsu.edu}.}}

\date{}

\begin{document}
\maketitle

\begin{abstract}
     We propose a finite-dimensional control-based method to approximate solution operators for evolutional partial differential equations (PDEs), particularly in high-dimensions. By employing a general reduced-order model, such as a deep neural network, we connect the evolution of the model parameters with trajectories in a corresponding function space. Using the computational technique of neural ordinary differential equation, we learn the control over the parameter space such that from any initial starting point, the controlled trajectories closely approximate the solutions to the PDE. Approximation accuracy is justified for a general class of second-order nonlinear PDEs. Numerical results are presented for several high-dimensional PDEs, including real-world applications to solving Hamilton-Jacob-Bellman equations. These are demonstrated to show the accuracy and efficiency of the proposed method.
     \medskip


\end{abstract}

\section{Introduction}
\label{sec:intro}
%
%
Partial Differential Equations (PDEs) are important in numerous disciplines from finance, to engineering, to science \cite{evans1998partial,renardy2006introduction,strauss2007partial}. As the solutions of many PDEs lack analytical form, it is necessary to use numerical methods to approximate them \cite{atkinson1989introduction, ames2014numerical,thomas2013numerical,johnson2012numerical,quarteroni2008numerical}. However, the traditional numerical methods, such as finite difference and finite element methods, rely upon the discretization of problem domains, which does not scale to high-dimensional problems due to the well-known issue of ``curse of dimensionality''.

This drawback has led to the more recent interest in using deep neural networks (DNNs) for the solving of PDEs in high-dimensions \cite{raissi2019physics-informed, bao2020numerical,han2017overcoming,e2018deep,cuomo2022scientific,han2018high-pde,han2018deep,kumar2011multilayer}. One of the most notable approaches \cite{raissi2019physics-informed, bao2020numerical,e2018deep,cuomo2022scientific,zang2020weak} is to parameterize the solution of a given PDE as a DNN, and the network parameters are trained to minimize potential violations (in various definitions) to the PDE. 
These methods have enjoyed success in solving a large variety of PDEs empirically. Theoretically, this has been partly justified by the provable universal approximation power of DNNs  \cite{hornik1991approximation, yarotsky2017error, liang2017deep}. 
Nonetheless, most deep-learning based PDE solvers aim to solve specific instances of PDEs, and as a consequence, they need to relearn from scratch whenever the initial and/or boundary value changes even for the same PDE operator. 

To address the aforementioned issue, there has been a collection of recent studies to find \emph{solution operators} of PDEs \cite{li2020fourier,lu2019deeponet,gaby2023neural}. A solution operator of a given PDE is the mapping from the problem's parameters (e.g., PDE coefficient, initial value, or boundary value) to the corresponding solution. Finding solution operators has substantial applications, especially in those where solutions to the same PDE are needed for different initial or boundary value configurations. 
Examples include (decentralized) robotic control problems where robotic agents need to plan optimal trajectories in response to frequently changing terminal cost functions. These agents often have limited computation and battery power and thus they can only adopt low-cost solution operator approximation implemented in their systems.

In this paper, we propose a novel control-based strategy to approximate solution operators of high-dimensional evolution PDEs based on the principle established in \cite{gaby2023neural}, which results in substantially improved solution efficiency and accuracy compared to the latter. 
Specifically, we parameterize the PDE solution using a reduced-order model (e.g., a deep neural network) and learn a control field to steer the model parameters, such that the rendered parameter trajectory induces an approximate solution of the PDE for any initial value. To this end, we adapt the computation technique developed in \cite{chen2018neural} to approximate this control field. 
After learning the control, we can approximate the PDE solution by integrating the control field given any initial condition, which is a significant computational improvement over existing methods \cite{gaby2023neural}. We remark that our new framework is purely based on the given evolution PDE and does not require any spatial discretization nor solutions of the PDE for training. This allows promising applications to high-dimensional PDEs.
We summarize the novelty and contributions of this work as follows:
\begin{enumerate}
    \item We develop a new control-based strategy to approximate the control field in parameter space for solution operator approximation. This strategy substantially improves upon training efficiency and approximation accuracy.
    
    \item We provide new theoretical analysis to establish the existence of neural networks that approximate the solution operator, as well as establish error bounds for the proposed method when solving a general class of nonlinear PDEs compared to \cite{gaby2023neural}.
    
    \item We demonstrate promising results on a variety of nonlinear high-dimensional PDEs, including Hamilton-Jacobi-Bellman equations for stochastic control, to show the performance of the proposed method. 
\end{enumerate}

The remainder of this paper is organized as follows. 
In Section \ref{sec:related}, we provide an overview of recent neural network based numerical methods for solving PDEs. We outline the fundamentals of our proposed approach and provide details of our method and its key characteristics in Section \ref{sec:method}. We show how under mild assumptions there exists a control vector field that achieves arbitrary accuracy and conduct comprehensive error analysis in Section \ref{sec:theory}. We demonstrate the performance of the proposed method on several linear and nonlinear evolution PDEs in Section \ref{sec:numerical-results}. Finally, Section \ref{sec:conclusion} concludes this paper.

\section{Related Work}
\label{sec:related}

\paragraph{Classical and neural-network methods for solving PDEs}
Classical numerical methods for solving PDEs, such as finite difference \cite{thomas2013fdm} and finite element methods \cite{johnson2012numerical}, often rely on discretizing the spatial domain using either a mesh or a triangularization. These methods have been significantly studied in the past decade \cite{ames2014numerical,thomas2013numerical,evans2012numerical,quarteroni2008numerical} and can handle complicated problems such as PDEs on irregular domains. However, they severely suffer the ``curse of dimensionality'' when applied to high-dimensional problems---the number of unknowns increases exponentially fast with respect to the spatial dimension, which renders them computationally intractable for many problems.

Using neural networks as solution surrogates for PDEs can be dated back to many early works \cite{dissanayake1994neural-network-based,lagaris1998artificial,lee1990neural,kumar2011multilayer}.
In recent years, DNNs have emerged as new powerful methods for solving PDEs through various numerical approaches 
\cite{raissi2019physics-informed,e2018deep,bao2020numerical,zang2020weak,sirignano2018dgm:,nusken2021solving,yang2020physics,berg2018unified}.
Of particular interest, DNNs have demonstrated extraordinary potential in solving many high-dimensional nonlinear PDEs, which for classical methods had long been considered computationally intractable. For example, a variety of DNN based methods have been proposed based on the strong form \cite{raissi2019physics-informed,nabian2018deep,dissanayake1994neural-network-based,berg2018unified,magill2018neural,pang2019fpinns,kharazmi2020hp,pang2020npinns,ramabathiran2021spinn}, variational form \cite{e2018deep}, and weak form \cite{zang2020weak,bao2020numerical} of PDEs.
They have been considered with adaptive collocation strategies \cite{anitescu2019artificial}, adversarial inference procedures \cite{yang2019adversarial}, oscillatory solutions \cite{cai2020phase}, and multiscale methods \cite{liu2020multi,wang2020multi,cai2019multi}.
Improvements in these methods have been accomplished by using adaptive activation functions \cite{jagtap2020adaptive}. Further, network structures \cite{gu2020selectnet,gu2020structure,huang2020int}, boundary conditions \cite{lyu2020enforcing,dong2020method}, structure probing \cite{huang2020int}, as well as their convergence \cite{luo2020two,shin2020convergence}, have also been studied. Readers interested in these methods can refer to \cite{ramabathiran2021spinn,yang2020physics,wang2022respecting,liang2022finite,wang2022isl2,zhang2020physics}.
Further, some methods can solve inverse problems such as parameter identification.

For a class of high-dimensional PDEs that have equivalent backward stochastic differential equation (SDE) formulations due to Feynman-Kac formula, have also seen deep learning methods applied by leveraging such correspondences \cite{beck2017machine,fujii2017asymptotic,han2017overcoming,e2017deep,han2018high-pde, han2020solving,pham2021neural,hure2020deep,hutzenthaler2020proof}.
These methods are shown to be good even in high dimensions \cite{han2018high-pde,hure2020deep,pham2021neural}, however, they are limited to solving the special type of evolution equations whose generator functions have a corresponding SDE.

For evolution PDEs, parameter evolution algorithms \cite{du2021evolutional, bruna2022neural-galerkin,anderson2022evolution} have also been considered. These methods generalize on Galerkin methods and parameterize the PDE solution as a neural network \cite{du2021evolutional,bruna2022neural-galerkin} or an adaptively chosen ansatz as discussed in \cite{anderson2022evolution}. In these methods, the parameters are evolved forward in time through a time marching scheme, where at each step a linear system \cite{bruna2022neural-galerkin,du2021evolutional} or a constrained optimization problem \cite{anderson2022evolution} needs to be solved to generate the values of the parameters at the next time step.

\paragraph{Learning solution operator of PDEs}
The previously discussed methods all share a similarity in that they aimed at solving specific instances of a given PDE. Therefore, they need to be rerun from scratch when any part of the problem configuration (e.g., initial value, boundary value, problem domain) changes. In contrast, the solution operator of a PDE can directly map a problem configuration to its corresponding solution. When it comes to learning solution operators, several DNN approaches have been proposed. 

One approach attempts to approximate Green's functions for some linear PDEs \cite{boulle2022learning-green,teng2022green,nicolas2022data-driven,lin2022bi-green}, as solutions to such PDEs have explicit expression based on their Green's functions.
However, this approach only applies to a small class of linear PDEs whose solution can be represented using Green's functions.
Moreover, Green's functions have singularities and special care is needed to approximate them using neural networks.
For example, rational functions are used as activation functions of DNNs to address singularities in \cite{nicolas2022data-driven}. In \cite{boulle2022learning-green}, the singularities are represented with the help of fundamental solutions. 

For more general nonlinear PDEs, DNNs have been used for operator approximation and meta-learning for PDEs
\cite{mao2020deepm,guo2018data-driven,lu2019deeponet,lu2019deepxde,li2020fourier,wen2022u-fno,wang2021learning,regazzoni2019machine}.
For example, the work \cite{guo2018data-driven} considers solving parametric PDEs in low-dimension ($d\le 3$ for the examples in their paper). Their method requires discretization of the PDE system and needs to be supplied by many full-order solutions for different combinations of time discretization points and parameter selections for their network training. Then their method applies proper orthogonal decomposition to these solutions to obtain a set of reduced bases to construct solutions for new problems. This can be seen as using classical solutions to develop a data-driven solution operator.
The work \cite{regazzoni2019machine} requires a massive amount of pairs of ODE/PDE control and the corresponding system outputs, which are produced by solving the original ODE/PDE system; then the DNN is trained on such pairs to learn the mapping between these two subjects which necessarily are discretized as vectors under their proposed framework.

In contrast, DeepONets \cite{lu2019deeponet,lu2019deepxde,wang2021learning} seek to approximate solution mappings by use of a ``branch" and ``trunk" network. FNOs \cite{li2020fourier,wen2022u-fno} use Fourier transforms to map a neural network to a low dimensional space and then back to the solution on a spatial grid. 
In addition, several works apply spatial discretization of the problem or transform the problem domains and use convolutional neural networks (CNNs) \cite{raonic2023convolutional,guo2016cnn-operator,zhu2018bayesian} or graph neural networks (GNNs) \cite{kovachki2023neural,li2020neural,lotzsch2022learning} to create a mapping from initial conditions to solutions. 
Interested readers may also refer to generalizations and extensions of these methods in \cite{chen2020meta,fan2019bcr,li2020neural,cai2020deepm,clark2020deep,mao2020deepm,pang2020npinns,lu2019deepxde,kovachki2021universal}.
A key similarity of all these methods is they require certain domain discretization and often a large number of labeled pairs of IVP initial conditions (or PDE parameters) and the corresponding solution obtained through other methods for training. This limits their applicability for high dimensional problems where such training data is unavailable or the mesh is prohibitive to generate due to the curse of dimensionality. 

The work \cite{gaby2023neural} develops a new framework to approximate the solution operator of a given evolution PDE by parameterizing the solution as a neural network (or any reduced-order model) and learning a vector field that determines the proper evolution of the network's parameters. Therefore, the infinite-dimensional solution operator approximation reduces to finding a vector field in the finite-dimensional parameter space of the network. In \cite{gaby2023neural}, this vector field is obtained by solving a least squares problem using a large number of sampled network parameters in the parameter space.

\paragraph{Differences between our proposed approach and existing ones}
Our approach follows the framework proposed in \cite{gaby2023neural} and thus allows solution operator approximation for high-dimensional PDEs and does not require any spatial discretization nor numerous solution examples of the given PDE for training. These avoid the issues hindering applications of all the other existing methods (e.g., DeepONet \cite{lu2019deeponet} and FNO \cite{li2020fourier}). 

On the other hand, the approach in this work improves both approximation accuracy and efficiency over \cite{gaby2023neural} by leveraging a new training strategy of control vector field in parameter space based on Neural ODE (NODE) \cite{chen2018neural}. In particular, this new approach only samples initial points in the parameter space and generates massive informative samples along the trajectories automatically during training, and the optimal parameter of the control vector field is learned by minimizing the approximation error of the PDE along these trajectories. This avoids both random sampling in a high-dimensional parameter space and solving expensive least squares problems as in \cite{gaby2023neural}. As a result, the new approach demonstrates orders of magnitudes higher accuracy and faster training compared to \cite{gaby2023neural}. Moreover, we develop a new error estimate to handle a more general class of nonlinear PDEs than \cite{gaby2023neural} does. Both the theoretical advancements and numerical improvements will be demonstrated in the present paper.

\section{Proposed Method}
\label{sec:method}

\subsection{Solution operator and its parameterization}
We follow the problem setting in \cite{gaby2023neural} and let $\Omega$ be an open bounded set in $\mathbb{R}^{d}$ and $F$ a \emph{(possibly) nonlinear differential operator} of functions $u: \Omega \to \Rbb$ with necessary regularity conditions, which will be specified below. We consider the IVP of the evolution PDE defined by $F$ with arbitrary initial value as follows:
\begin{equation}
\begin{cases}
\partial_t u(x,t) = F [u](x,t), & \ x \in \Omega,\ t \in (0,T],\\
u(x,0)=g(x), & \ x \in \Omega,
\end{cases}
\label{eq:pde}
\end{equation}
where $T>0$ is some prescribed terminal time, and $g:\Rbb^d \to \Rbb$ stands for an initial value. For ease of presentation, we assume $u$ to be compactly supported in $\Omega$ (for compatibility we henceforth assume $g(x)$ has zero trace on $\partial \Omega$) throughout this paper.
We denote $u^{g}$ the solution to the IVP \eqref{eq:pde} with this initial $g$. 
The solution operator $\Scal_{F}$ of the IVP \eqref{eq:pde} is thus the mapping from the initial $g$ to the solution $u^{g}$ :
\begin{equation}
\label{eq:so}
    \Scal_{F}: C(\bar{\Omega}) \cap C^2(\Omega) \to C^{2,1}(\bar{\Omega}\times[0,T]), \quad \mbox{such that} \quad g \mapsto \Scal_{F}(g) := u^{g}.
\end{equation}
%
%
Our goal is to find a numerical approximation to $\Scal_{F}$. Namely, we want to find \emph{a fast computational scheme $\Scal_{F}$ that takes any initial $g$ as input and accurately estimate $u^{g}$ with low computation complexity.}
Specifically, we expect this scheme $\Scal_{F}$ to satisfy the following properties:
\begin{enumerate}
    \item The scheme applies to PDEs on high-dimensional $\Omega \subset \Rbb^{d}$ with $d\ge 5$.
    \item The computation complexity of the mapping $g \mapsto \Scal_{F}(g)$ is much lower than solving the problem \eqref{eq:pde} directly.
\end{enumerate}

It is important to note that the second item above is due to the substantial difference between solving \eqref{eq:pde} for any given but fixed initial value $g$ and finding the solution operator \eqref{eq:so} that maps any $g$ to the corresponding solution $u^{g}$.
In the literature, most methods belong to the former class, such as finite difference and finite element methods, as well as most state-of-the-art machine-learning based methods.
However, these methods are computationally expensive if \eqref{eq:pde} must be solved with many different initial values, either in parallel or sequentially, and they essentially need to start from scratch for every new $g$.
In a sharp contrast, our method belongs to the second class in order to approximate the solution operator $\Scal_{F}$ which, once found, can help us to compute $u^{g}$ for any given $g$ at much lower computational cost.

To approximate the solution operator $\Scal_{F}$ in \eqref{eq:so}, we follow the strategy devleoped in \cite{gaby2023neural} and develop a new control mechanism in the parameter space $\Theta$ of a prescribed reduced-order model $u_{\theta}$.
%
%
Specifically, we first determine select a model structure $u_{\theta}$ (e.g., a DNN) to represent solutions of the IVP. 
We assume that $u_{\theta}(x) := u(x;\theta)$ is $C^1$ smooth with respect to $\theta$.
This is a mild condition satisfied by many different parameterizations, such as all typical DNNs with smooth activation functions.
%
%
Suppose there exists a trajectory $\{\theta(t): \, 0 \le t \le T\}$ in the parameter space, we only need
%
%
\begin{equation}
\nabla_{\theta}u_{\theta(t)}(x) \cdot \dot{\theta}(t)  = \partial_t (\utt(x)) = F[\utt](x), \qquad \forall\, x \in \Omega,\ t \in (0,T] 
\label{eq:nn-pde}
\end{equation}
and the initial $\theta(0)$ to satisfy $u_{\theta(0)}=g$. The first and second equalities of \eqref{eq:nn-pde} are due to the chain rule and the goal for $u_{\theta(t)}(\cdot)$ to solve the PDE \eqref{eq:pde}, respectively.
Here we use $\nabla_{\theta}$ to denote the partial derivative with respect to $\theta$ (and $\nabla$ is the partial derivative with respect to $x$).
To achieve \eqref{eq:nn-pde}, \cite{gaby2023neural} proposed to learn a DNN $V_{\xi}$ with parameters $\xi$ by solving the following nonlinear least squares problem:
\begin{equation}
    \label{eq:ls-vf}
    \min_{\xi} \int_{\Theta} \int_{\Omega} |\nabla_{\theta} \ut(x) \cdot V_{\xi}(\theta) - F[\ut](x)|^2 \,dx d\theta.
\end{equation}
%
%
%
%
Once $\xi$ is obtained, one can effectively approximate the solution of the IVP with any initial value $g$: first find $\theta(0)$ by fitting $u_{\theta(0)}$ to $g$ (e.g., $\theta(0) = \argmin_{\theta} \|\ut - g\|_{2}^2$, which is fast to compute)
and then numerically integrate $\dot{\theta}(t) = V_{\xi}(\theta(t))$ (which is again fast) in the parameter space $\Theta$.
%
%
The solution trajectory $\{\theta(t):\, 0 \le t \le T\}$ induces a path $\{\utt:\, 0 \le t \le T\}$ as an approximation to the solution of the IVP. 
The total computational cost is substantially lower than solving the IVP \eqref{eq:pde} directly as shown in \cite{gaby2023neural}.


\subsection{Proposed control field approximation}

The approximation strategy developed above is promising \cite{gaby2023neural}, however, the computational cost in solving the nonlinear least squares problem \eqref{eq:ls-vf} is high due to the relatively high dimension $m$ of $\Theta$. Therefore, it requires some intuition and careful tuning in sampling $\theta$ from such $\Theta$ to achieve reasonable solution quality. These issues may hinder the practical applications of this approach.

In this paper, we proposed a new approach to estimate $\xi$ which is substantially more efficient and accurate, even without any intuition, in training $V_{\xi}$ than using \eqref{eq:ls-vf}.
%
%
%
Specifically, we consider $\theta(t)$ as a controllable trajectory and notice that it suffices to find a proper control vector field $V_{\xi}$ with parameters $\xi$ that minimizes the following running cost
\begin{equation}
\label{eq:running_cost}
    r(\theta(t) ; \xi) := \|\nabla_{\theta}\utt \cdot V_{\xi}(\theta)- F[\utt]\|_{L^{2}(\Omega)}^{2}
\end{equation}
for any trajectory $\theta(t)$ in $\Theta$. These trajectories can be unevenly distributed in $\Theta$ for the corresponding $\utt$ to be meaningful solutions to \eqref{eq:pde}, which is the major issue causing inefficient sampling of $\theta$ in \eqref{eq:ls-vf}. Therefore, we propose the following strategy to learn $V_{\xi}$ effectively: first, define the auxiliary variable $\gamma(t):=[\theta(t); s(t)] \in \Rbb^{m+1}$ (we use the MATLAB syntax $[\cdot;\cdot]$ to stack vectors as one column vector) such that $\gamma(0)=[\theta(0);0]$ and
\[
\dot{\gamma}(t)=[\dot{\theta}(t);\ \dot{s}(t)] = 
    [V_{\xi}(\theta(t));\ r(\theta(t); \xi)].
\]
Define $\ell(\gamma; \xi)=[0_{m};1]^{\top}\gamma$, and we see $\ell(\gamma(T);\xi)=\int_0^{T}r(\theta(t);\xi)dt$. We then sample $M$ initial values $\{g_i:\ i=1,\dots,M\}$ and fit them by $\{u_{\theta_i(0)}:\ i=1,\dots,M\}$ correspondingly, then solve the following terminal value control problem with control parameters $\xi$:
\begin{equation}
    \label{eq:many_control_problem}
    \begin{aligned}
        \min_{\xi}\quad &\hat{\ell}(\xi) := \frac{1}{M}\sum_{i=1}^M\ell(\gamma_i(T); \xi),\\
        \text{subject to}\quad & 
        \dot{\gamma}_i(t)= [V_{\xi}(\theta_i(t));\ r(\theta_i(t); \xi)],\quad  \gamma_i(0)=[\theta_i(0);\ 0],\qquad i=1,\ldots,M, 
    \end{aligned}
\end{equation}
where each $\theta_i(t)$ starts from its own initial $\theta_i(0)$ fitted earlier.
In practice, we shall be required to approximate the $L^2$ norm within $r(\theta;\xi)$ in \eqref{eq:running_cost} using Monte Carlo integration as usual. 

With the formulation in \eqref{eq:many_control_problem}, we now show how to calculate gradients of $\hat{\ell}$ using Neural ODE (NODE) \cite{chen2018neural}.
It suffices to consider the case with $M=1$.
NODE allows us to calculate $\nabla_{\xi}\ell (\gamma(T),\xi)$ as follows: using the \textit{adjoint sensitivity method}, we introduce the so-called adjoint parameter $a:[0,T]\to \Rbb^{m+1}$ such that $a(T)=-\nabla_{\gamma}\ell(\gamma(T);\xi)$ and $\dot{a}(t)^{\top}=-a(t)^{\top}\nabla_{\gamma} [V_{\xi}(\theta(t)); r(\theta(t);\xi)].$
%
%
In practice, we compute $[\theta(t); r(\theta(t);\xi)]$ forward in time to $T$, and then compute $a(t)$ and integrate $a(t)^{\top}\nabla_{\xi} [V_{\xi}(\theta(t)); r(\theta(t);\xi)]$ simultaneously backward in time to calculate $\nabla_{\xi}\ell(\gamma(T),\xi) = -\int_T^0 a(t)^{\top} \nabla_{\xi}[V_{\xi}(\theta(t)); r(\theta(t);\xi)]dt$. The latter can be used in any standard first-order optimization algorithm, such as stochastic gradient descient method and ADAM \cite{kingma2015adam:} to minimize $\ell$.
This can be readily extended to the case with $M>1$ and averaged loss $\hat{\ell}$ in \eqref{eq:many_control_problem}.
The training of $V_{\xi}$ is summarized in Algorithm \ref{alg:neural-control}. 
In this work, we apply ADAM to minimize $\hat{\ell}$ and its parameter settings are provided in Section \ref{sec:numerical-results}.


\begin{algorithm}
\caption{Training neural control $V_{\xi}$}
\label{alg:neural-control}
\begin{algorithmic}[1]
\Require{Reduced-order model structure $\ut$ and parameter set $\Theta$. Control vector field structure $V_{\xi}$. Error tolerance ${\varepsilon}$.}
\Ensure{Optimal control parameters ${\xi}$.}
\State Sample $\{\theta_{k}(0)\}_{k=1}^{K}$ uniformly from $\Theta$. 
\State Form loss $\hat{\ell}(\xi)=\frac{1}{M}\sum_{i=1}^M\int_0^{T} \ell(\theta_i(t);\xi)dt$ as in \eqref{eq:many_control_problem}.
\State Minimize $\hat{\ell}(\xi)$ with respect to $\xi$ (using Neural ODE to compute $\nabla_{\xi}\hat{\ell}(\xi)$).
\end{algorithmic}
\end{algorithm}

Once we trained the vector field $V_{\xi}$, we can implement the solution operator $\Scal_{F}$ in the following two steps: we first find a $\theta(0)$ such that $u_{\theta(0)}$ fits $g$, i.e., find $\theta(0)$ that minimizes
$\| u_{\theta} - g \|_2$. This can be done by sampling $\{x_{n}\}_{n=1}^{N}$ from $\Omega$ and minimizing the empirical squared $L^2$ norm $(1/N)\cdot \sum_{n=1}^{N} |\ut(x_n) - g(x_n)|^2$ with respect to $\theta$. Then we solve the ODE 
\begin{equation}
    \label{eq:ode}
    \dot{\theta}(t) = V_{\xi}(\theta(t))
\end{equation}
using any numerical ODE solver (e.g., Euler, 4th order Runge-Kutta, predictor-corrector) with $\theta(0)$ as the initial value. Both steps can be done quickly. We summarize how neural control solves IVPs in Algorithm \ref{alg:solving-ivp}.

\begin{algorithm}[t]
\caption{Implementation of solution operator $\Scal_{F}$ of the IVP \eqref{eq:pde} using trained control $V_{\xi}$}
\label{alg:solving-ivp}
\begin{algorithmic}[1]
\Require{Initial value $g$ and tolerance $\varepsilon_{0}$. Reduced-order model $\ut$ and trained neural control $V_{\xi}$.}
\Ensure{Trajectory $\theta(t)$ such that $u_{\theta(t)}$ approximate the solution $\Scal_{F}[g]$ of the IVP \eqref{eq:pde}.}
\State{Compute initial parameters $\theta(0)$ such that $\|u_{\theta(0)} - g \|_2 \le \varepsilon_{0}$.}
\State{Use any ODE solver to compute $\theta(t)$ by solving \eqref{eq:ode} with approximate field $V_{\xi}$ and initial $\theta(0)$.}

\end{algorithmic}
\end{algorithm}

\section{Theoretical Advances}
\label{sec:theory}

In this section, we provide theoretical justifications on the solution approximation ability of the proposed method for a general class of second-order semi-linear PDEs. For ease of presentation, we first define Sobolev ball with radius $L$ in $W^{k,\infty}(\Omega)$ as follows.
\begin{definition}[Sobolev ball]
    \label{def:sobolev_ball}
    For $L>0$ and $k \in \mathbb{N}$, we define the \text{Sobolev ball} of radius $L$ as the set 
    \[
    SB(\Omega,L,k):=\{g \in W^{k,\infty}(\Omega): \ \ \|g\|_{W^{k,\infty}(\Omega)}\leq L\}.
    \]
\end{definition}
As mentioned in Section \ref{sec:method}, we focus on IVPs with solutions compactly supported in an open bounded set $\Omega \subset \Rbb^d$ for ease of discussion. More specifically, we consider the following IVP with an evolution PDE and arbitrary initial value $g$:
\begin{equation}
\label{eq:pde_g}
    \begin{cases}
    \partial_t u(x,t)=F[u](x,t) & x \in \Omega, \ t \in (0,T)\\
    u(x,0) = g(x) & x \in \Omega,\\
    u(x,t) = 0 & x \in \partial\Omega, \ t \in [0,T]
    \end{cases}
\end{equation}
where $F$ is a $k$th order autonomous, (possibly) nonlinear differential operator. We note that non-autonomous PDEs can be converted to an equivalent autonomous one by augmenting spatial variable $x$ with $t$ in practice. For the remainder of this section, we require the following regularity on $F$.
\begin{assumption}
\label{assump:regularity}
 Let $F$ be a $k$th order differential operator and $p \in \mathbb{N}$. For all $L>0$, there exists $M_{p,L,F}>0$ such that for all $w,v \in  SB(\Omega,L,k+p)$ we have
\[
\|F[w]-F[v]\|_{W^{p,\infty}(\Omega)}\leq M_{p,L,F} \|w-v\|_{W^{k+p,\infty}(\Omega)}
\]
\end{assumption}
The requirements of Assumption \ref{assump:regularity} ensure that $F: W^{k+p,\infty}(\Omega) \to W^{p,\infty}(\Omega)$ is Lipschitz over $SB(\Omega,L,k+p)$ for the specified $p$. For many types of PDEs, this assumption is quite mild. For example, if the second-order linear/nonlinear differential operator $F$ is $C^p$ in $(x,u(x),\nabla u(x), \nabla^2u(x))$ for all $x \in \Omega$, then $F$ satisfies Assumption \ref{assump:regularity} with $k=2$.

\begin{theorem}
\label{thm:theta_exist}
Suppose $F$ satisfies Assumption \ref{assump:regularity} for $k,p \in \Nbb$. Let $\epsilon>0$ and $L>0$. Then there exists a feed-forward neural network $\ut$ and an open bounded set $\Theta_{u,F,L} \subset \Rbb^m$ such that
\begin{enumerate}
    \item[(i)] For every $\theta \in \Theta_{u,F,L}$, there exists $\alpha_{\theta} \in \Rbb^m$ such that
    \[
    \|\alpha_{\theta} \nabla_{\theta}\ut-F[\ut]\|_{W^{1,\infty}(\Omega)}< \epsilon.
    \]
    \item[(ii)] For every $g \in W^{2k+3,\infty}(\Omega) \cap SB(\Omega,L,2k+2)$, there exists $\theta \in \Theta_{u,F,L}$ such that
    \[
    \|\ut - g \|_{W^{1,\infty}(\Omega)}< \epsilon.
    \]
\end{enumerate}
\end{theorem}

The proof of Theorem \ref{thm:theta_exist} can be found in Appendix \ref{appx:A}.
Theorem \ref{thm:theta_exist} tells us that with a sufficiently large neural network, we can find a subset $\Theta_{u,F,L}$ of the parameter space such that the tangent spaces of $\Theta_{u,F,L}$ can approximate $F[u_{\theta}]$ arbitrarily well. This is crucial if our method is going to be able to accurately track the solutions of \eqref{eq:pde} along $t$. With Theorem \ref{thm:theta_exist} we can now state the following result on the existence of vector fields defined over $\Theta_{u,F,L}$ that have the approximation capabilities proved in Theorem \ref{thm:theta_exist}. This proposition is due to \cite{gaby2023neural} which had a form of Theorem \ref{thm:theta_exist} used as an assumption in the proof.
\begin{proposition}
Suppose Assumption \ref{assump:regularity} holds. Then for any $\varepsilon > 0$ and $L>\epsilon>0$, there exists a differentiable vector field parameterized as a neural network $V_{\xi}: \bar{\Theta}_{u,F,L} \to \Rbb^{m}$ with parameters $\xi$, such that 
\[
\|V_{\xi}(\theta) \cdot \nabla_{\theta}\ut-F[\ut]\|_{H^1(\Omega)} \leq \varepsilon,
\]
for all $\theta \in \bar{\Theta}_{u,F,L}$.
\label{prop:F_exists}
\end{proposition}
The proof of Proposition \ref{prop:F_exists} can be found in Appendix \ref{appx:B}.
Now we provide a few mild conditions on the problem setting of the PDE and its solutions, as follows.
\begin{assumption}
\label{assummp:compact}
    Let $\Omega$ be a bounded open set of $\Rbb^d$. The solutions to \eqref{eq:pde} are compactly supported within $\Omega$ for all $t \in [0,T]$.
\end{assumption}
Of particular importance, Assumption \ref{assummp:compact} allows us to be able to design $u_{\theta}$ so that $\nabla u_{\theta}(x)=0$ and $u_{\theta}(x)=0$ on the boundary and as such matches the true solutions on the boundaries as well (see e.g. \cite{gaby2023neural,du2021evolutional,bruna2022neural-galerkin} for more discussion on this point).
\begin{assumption}
\label{assump:operator}
    Let \eqref{eq:pde} be a second-order semi-linear PDE with $F[u](x):=(\sigma^2(x)/2)\Delta u(x)+f(x,u(x),\nabla u(x))$ for some (possibly) fully nonlinear function $f$ such that the following inequalities hold for all $x,y \in \Omega$, $u,v \in \Rbb$ and $p,q\in \Rbb^d$:
    \begin{subequations}
    \label{eq:lipschitz_op}
        \begin{align}
        |f(x,u,p)-f(y,v,q)| & \leq L_f(|x-y|+|u-v|+|p-q|) \label{eq:lipschitz_op_L}\\
        |\partial_x f(x,u,p)-\partial_x f(y,v,q)| & \leq L_x(|x-y|+|u-v|+|p-q|) \label{eq:lipschitz_op_Lx}\\
        |\partial_u f(x,u,p)-\partial_u f(y,v,q)| & \leq L_u(|x-y|+|u-v|+|p-q|) \label{eq:lipschitz_op_Lu}\\
        |\partial_p f(x,u,p)-\partial_p f(y,v,q)| & \leq L_p(|x-y|+|u-v|+|p-q|) \label{eq:lipschitz_op_Lp}
    \end{align}
    \end{subequations}
    for some $L_f,L_x,L_u,L_p>0$. Here $|\cdot|$ denotes absolute value or standard Euclidean vector norm.
\end{assumption}
It should be noted that $F$ as defined in Assumption \ref{assump:operator} satisfies the requirements of Assumption \ref{assump:regularity}, and thus we can use the results of Proposition \ref{prop:F_exists}. 
Now we are ready to present the main theorem of the present paper.
\begin{theorem}
Suppose Assumptions \ref{assummp:compact} and \ref{assump:operator} hold. Let $\varepsilon_0 \ge \varepsilon >0$ be arbitrary. Then for any $L>\varepsilon>0$ there exist a vector field $V_{\xi}$ and a constant $C>0$ depending only on $F$, $u_{\theta}$, and $V_{\xi}$, such that for any $u^* \in SB(\Omega,L,2)$ satisfying the evolution PDE in \eqref{eq:pde}, Assumption \ref{assummp:compact} and $u^*(\cdot,0)\in SB(\Omega,L,4)$ there is
\begin{equation}
\label{eq:utheta-ustar-L2}
    \|u_{\theta(t)}(\cdot)-u^*(\cdot,t)\|_{H^{1}(\Omega)}\leq\sqrt{2}(\varepsilon_{0} + 2{\varepsilon} t)e^{4Ct}
\end{equation}
for all $t$ as long as $\theta(t) \in \barTheta_{u,F,L}$, where $\theta(t)$ is solved from the ODE \eqref{eq:ode} with $V_{\xi}$ and initial $\theta(0)$ satisfying $\|u_{\theta(0)}(\cdot)-u^*(\cdot,0)\|_{L^{2}(\Omega)}\leq \varepsilon_0$. 
\label{thm:ode-sol-approx-error}
\end{theorem}
\begin{proof}
Due to Assumptions \ref{assump:operator} and \ref{assump:regularity}, there exist a vector field $V_{\xi}$ and a bounded open set $\Theta_{u,F,L}$ that satisfy the results of Proposition \ref{prop:F_exists}. 
Assume $\theta(t)\in \Theta_{u,F,L}$ is the solution to \eqref{eq:ode}, and $\|u_{\theta(0)}(\cdot)-u^*(\cdot,0)\|_{2}\leq \varepsilon_0$ (such $\theta(t)$ is guaranteed to exist for $\varepsilon\leq \varepsilon_0$), where $u^*$ is the true solution of the PDE \eqref{eq:pde}. As $u^* \in SB(\Omega,L,2)$, there exists $M_f:=\|\partial_u f \|_{L^{\infty}(\Omega \times [-L,L] \times [-L,L]^d)}$ because $u^* \in SB(\Omega,L,2)$ which implies $\partial_u f(x,u^*(x,t),\nabla u^*(x,t)) \leq M_f$ a.e.\ in $\Omega$ for every $t$. Similarly, by Assumption \ref{assump:operator} we have $\partial_p f$ Lipschitz continuous jointly in its arguments and thus there is $\nabla \cdot \partial_p f(x,u^*(x,t),\nabla u^*(x,t)) \leq M_{f}'$ a.e.\ in $\Omega$ for some $M_f'>0$.

In what follows, we denote $\|\cdot\|_2 := \|\cdot\|_{L^{2}(\Omega)}$ and 
\[
u_t:=u_{\theta(t)}, \quad v_t:=u_t^*, \quad e_0(\cdot,t):=u_t(\cdot) - v_t(\cdot) , \quad e_1(\cdot,t):=\nabla u_t(\cdot) - \nabla v_t(\cdot)
\]
to simplify the notations hereafter. Define
\[
E(t)=\frac{1}{2}\int_{\Omega}|u_t(x) - v_t(x)|^2 \,dx + \frac{1}{2}\int_{\Omega}|\nabla u_t(x)-\nabla v_t(x)|^2\,dx= \frac{\|e_0(\cdot,t)\|_2^2}{2}+\frac{\|e_1(\cdot,t)\|_2^2}{2}.
\]
Taking the time derivative of $E$, we find
\begin{align}
    \dot{E}(t)&=\ \langle e_0(\cdot,t),\partial_t e_0(\cdot,t)\rangle + \langle e_1(\cdot,t), \partial_t e_1(\cdot,t)\rangle \nonumber \\
    &= \  \langle u_t-v_t,F[u_t] - F[v_t] \rangle + \langle \nabla u_t- \nabla v_t, \nabla F[u_t] - \nabla F[v_t] \rangle + \langle u_t-v_t, \epsilon_t \rangle + \langle \nabla u_t - \nabla v_t , \nabla \epsilon_t \rangle \nonumber \\
    &=:\  \text{I} + \text{II} + \text{III} + \text{IV}, \label{eq:Edot}
\end{align} 
where $\langle \cdot, \cdot \rangle$ is the inner product in the $L^2(\Omega)$ space, $\epsilon_t(\cdot):=\nabla_{\theta}u_{\theta(t)}(\cdot) \cdot V_{\xi}(\theta(t))-F[u_{\theta(t)}](\cdot)=\partial_t u_t(\cdot) - F[u_t](\cdot)$, and the terms III and IV result from $\partial_t u_t = F[u_t] + \epsilon_t$. We shall now proceed to estimate the four terms above.

For the term $\text{I}$, we first notice that
\begin{equation}
    \label{eq:I_1}
    \begin{aligned}
    \text{I}_1:=&\ \big\langle u_t-v_t,\frac{\sigma^2}{2}\left(\Delta u_t - \Delta v_t\right) \big\rangle\\
    =&\ \int_{\Omega} \frac{\sigma^2 (x)}{2}e_0(x,t)\Delta e_0(x,t) dx=\int_{\Omega}\frac{\sigma^2(x)}{2}\left(\nabla \cdot (e_0(x,t) \nabla e_0(x,t))-|\nabla e_0(x,t)|^2 \right)dx\\
    \leq &\ \int_{\partial \Omega} \frac{\sigma^2(x)}{2}e_0(x,t)\nabla e(x,t) \cdot \vec{n}(x)\,dS(x)
    -\int_{\Omega} (\sigma(x)\nabla \sigma(x)) \cdot (e_0(x,t)\nabla e_0(x,t)) dx\\
    \leq &\  M_{\sigma} \| e_0(\cdot,t)\|_2\|\nabla e_0(t, \cdot) \|_2,
    \end{aligned}
\end{equation}
where $M_{\sigma}:=\|\sigma\nabla \sigma\|_{L^{\infty}} < \infty$, $\vec{n}(x)$ is the unit outer normal at $x \in \partial \Omega$, and the first inequality is due to $\int_{\Omega} \frac{\sigma^2}{2}|\nabla e_0|^2 \ge 0$.
%
Furthermore, we can show a bound for the quantity:
\begin{equation}
    \label{eq:I_2}
    \begin{aligned}
        \text{I}_2:=& \langle e_0(\cdot,t), \fut - \fvt \rangle
        \leq L_f\|e_0(\cdot,t)\|_2(\|e_0(\cdot,t)\|_2 + \|\nabla e_0(\cdot,t)\|_2),
    \end{aligned}
\end{equation}
where we used the Lipschitz bound on $f$ to arrive at the inequality.
Realize $\text{I}=\text{I}_1+\text{I}_2$ and so
\begin{equation}
    \label{eq:I}
    \text{I} \leq (M_{\sigma}+L_f) \| e_0(\cdot,t)\|_2\|\nabla e_0(t, \cdot) \|_2 + L_f\|e_0(\cdot,t)\|_2^2.
\end{equation}

Next, consider the term
\begin{equation}
    \label{eq:II_2}
    \begin{aligned}
        \big\langle \nabla u_t - \nabla v_t, \nabla (\frac{\sigma^2}{2}\Delta u_t - \frac{\sigma^2}{2}\Delta v_t) \big\rangle 
        =&\ \int_{\partial \Omega} \frac{\sigma^2(x)}{2}\Delta e_0(x,t) \nabla e_0(x,t) \cdot \vec{n}(x) \,dS(x)\\
        &\ \qquad -\int_{\Omega} \frac{\sigma^2 (x)}{2}(\Delta e_0(x,t))^2dx\\
        \leq&\ 0,
    \end{aligned}
\end{equation}
where the inequality comes as $\nabla e_0(x,t)=0$ on the boundary by Assumption \ref{assummp:compact}. Now we define
\begin{equation}
    \label{eq:II_22}
    \begin{aligned}
        \text{II}_2 =&\ \langle \nabla u_t- \nabla v_t, \nabla [f(\cdot,u_t,\nabla u_t)-f(\cdot, v_t, \nabla v_t)]\rangle\\
        =& \ \langle \nabla e_0(\cdot,t), \partial_x f(\cdot, u_t, \nabla u_t)- \partial_x f(\cdot, v_t, \nabla v_t) \rangle \\
        &\ \quad +  \langle \nabla e_0(\cdot,t), (\partial_u f(\cdot, u_t, \nabla u_t) \nabla u_t-\partial_u f(\cdot, v_t, \nabla v_t)\nabla v_t) \rangle\\
        &\ \quad +  \langle \nabla e_0(\cdot,t), (\nabla^2 u_t\partial_p f(\cdot, u_t, \nabla u_t)-\nabla^2 v_t\partial_p f(\cdot, v_t, \nabla v_t)) \rangle\\
        =&\ \text{II}_{2,1}+\text{II}_{2,2}+\text{II}_{2,3}.
    \end{aligned}
\end{equation}
We now estimate these three terms using our assumptions on the function $f$. First
\[
    \begin{aligned}
        \text{II}_{2,1}\leq &\ L_x \int_{\Omega}\nabla e_0(x,t) (|u_t(x)-v_t(x)|+|\nabla u_t(x) - \nabla v_t(x)|)dx\\
        \leq &\ L_x(\|\nabla e_0(\cdot,t)\|_2\|e_0(\cdot,t)\|_2+\|\nabla e_0(\cdot,t)\|_2^2).
    \end{aligned}
\]
Since $\barTheta_{u,F,L}$ is compact and $\theta(t)$ is bounded, it is easy to see there exists a constant $M_u>0$ such that $\|u_t\|_{W^{2,\infty}(\Omega)}\leq M_u$ where $M_u$ depends only on $\barTheta_{u,F,L}$ and the architecture of $u_t$. Therefore,
\[
\begin{aligned}
    \text{II}_{2,2}=&\ \langle \nabla e_0(\cdot,t), \left(\partial_u \fut-\partial_u \fvt \right)\nabla u_t+\partial_u \fvt \nabla e_0(\cdot,t) \rangle\\
    \leq &\  M_u L_u\|\nabla e_0(\cdot,t)\|_2(\|e_0(\cdot,t)\|_2+\|\nabla e_0(\cdot,t)\|_2)+M_f\|\nabla e_0(\cdot,t)\|_2^2.
\end{aligned}
\] 
%
%
Finally, we have
\[
\begin{aligned}
    \text{II}_{2,3}=&\  \langle \nabla e_0(\cdot,t), \nabla^2 u_t\partial_p \fut-\nabla^2 v_t\partial_p \fvt \rangle\\
    = &\  \langle \nabla e_0(\cdot,t), \nabla^2 u_t(\partial_p \fut - \partial_p \fvt) \rangle + \langle \nabla e_0(\cdot,t), (\nabla^2 e_0(\cdot,t))\partial_p \fvt \rangle.
\end{aligned}
\] Now we see
\begin{equation}
    \label{eq:II_231}
    \begin{aligned}
    \langle \nabla e_0(\cdot,t), \nabla^2 u_t(\partial_p \fut - & \partial_p \fvt) \rangle \leq M_u L_p \|\nabla e_0(\cdot,t)\|_2(\|e_0(\cdot,t)\|_2+ \|\nabla e_0(\cdot,t)\|_2)
    \end{aligned}
\end{equation}
and
\begin{equation}
    \label{eq:II_232}
    \begin{aligned}
        \langle \nabla e_0(\cdot,t), (\nabla^2 e_0(\cdot,t))\partial_p \fvt \rangle 
        =&\  \int_{\Omega}\partial_p f(x,v_t(x),\nabla v_t(x))^{\top}\nabla^2 e_0(x,t) \nabla e_0(x,t)dx\\
        =&\ \int_{\Omega}\partial_p f(x,v_t(x),\nabla v_t(x))^{\top}\nabla\left(\frac{1}{2} |\nabla e_0(x,t)|^2 \right)dx\\
        =&\ -\frac{1}{2}\int_{\Omega}\nabla \cdot(\partial_p f(x,v_t(x),\nabla v_t(x)))|\nabla e_0(x,t)|^2dx\\
        \leq &\  \frac{d}{2}M_{f}'\|\nabla e_0(\cdot,t)\|_2^2,
    \end{aligned}
\end{equation}
where we used the fact that $\nabla e_0(\cdot,t)=0$ on $\partial \Omega$ to get $\int_{\partial\Omega} \partial_p f |\nabla e_0|^2 dS(x) = 0$.
Applying \eqref{eq:II_231} and \eqref{eq:II_232} in $\text{II}_{2,3}$, we find
\[
\text{II}_{2,3}\leq M_uL_p \|\nabla e_0(\cdot,t)\|_2\|e_0(\cdot,t)\|_2+\left(M_uL_p+\frac{d}{2}M_{f}'\right)\|\nabla e_0(\cdot,t)\|_2^2.
\]Define 
\[
C:= L_x+M_uL_u+M_uL_p+\max \Big\{M_{\sigma}+L_f, \ M_f+\frac{d}{2}M_{f}' \Big\},
\] and we can see
\begin{equation}
    \label{eq:IandII}
    \begin{aligned}
    \text{I} +\text{II} \leq &\ C(\|e_0(\cdot,t)\|_2^2+\|e_0(\cdot,t)\|_2\|\nabla e_0(\cdot,t)\|_2+\|\nabla e_0(\cdot,t)\|_2^2)\\
    \leq&\ 2C\left(\|e_0(\cdot,t)\|_2^2+\|\nabla e_0(\cdot,t)\|_2^2\right).
    \end{aligned}
\end{equation}

Now, using our bound on $\epsilon_t$, we can bound $\text{III}$ and $\text{IV}$:
\begin{equation}
    \label{eq:IIIandIV}
    \begin{aligned}
        \text{III}+\text{IV}=&\ \langle e_0(\cdot,t), \epsilon(\cdot,t)\rangle+\langle \nabla e_0(\cdot,t), \nabla \epsilon(\cdot,t)\rangle \\
        \leq &\  \bar{\varepsilon}(\|e_0(\cdot,t)\|_2+\|\nabla e_0(\cdot,t)\|_2)\\
         \leq &\ \sqrt{2}\bar{\varepsilon}\sqrt{\|e_0(\cdot,t)\|_2^2+\|\nabla e_0(\cdot,t)\|_2^2}.
    \end{aligned}
\end{equation}
Applying \eqref{eq:IandII} and \eqref{eq:IIIandIV} to \eqref{eq:Edot} and dividing both sides by $E(t)$, we see
\[
\frac{d}{dt}\sqrt{E(t)}=\frac{ \dot{E}(t)}{\sqrt{E(t)}}\leq 4C\sqrt{E(t)}+2\bar{\varepsilon}.
\] 
Applying Gr\"{o}nwall's inequality, we conclude
\[
\sqrt{E(t)} \leq (\sqrt{E(0)}+2\bar{\varepsilon}t)e^{4Ct}=(\varepsilon_0+2\bar{\varepsilon}t)e^{4Ct}.
\]
The above inequality holds for any $u^*(\cdot,0)\in SB(\Omega,L,4)$ by Proposition \ref{prop:F_exists}, and so the proof is completed.
\end{proof}

\begin{remark}
Theorem \ref{thm:ode-sol-approx-error} provides an upper bound on the error between $u_{\theta(t)}$ and $u^*(\cdot,t)$: it depends on the projection error $\varepsilon$ and initial error $\varepsilon_0$ linearly, while it grows exponentially fast in $t$ which is expected. Note that Theorem \ref{thm:ode-sol-approx-error} extends Theorem 3.6 in \cite{gaby2023neural} from PDEs that are linear in the first- and second-order terms to general semi-linear PDEs. This extension requires substantial efforts and new proof steps. 
\end{remark}

\section{Numerical Results}
\label{sec:numerical-results}

\subsection{Implementation}
\label{subsec:implementation}


Throughout all experiments, we select the following special form for our neural network $V_{\xi}: \Rbb^m \to \Rbb^m$, 
\begin{equation}
    \label{eq:V-network-def}
    V_{\xi}(\theta)=\phi^{(s)}_{\xi_s}(\theta)\left(\phi^{(r)}_{\xi_r}(\theta)+\phi^{(e)}_{\xi_e}(\theta)\odot \theta\right),
\end{equation}
where $\xi=(\xi_s,\xi_r,\xi_e)$, $\odot$ is component wise multiplication, and $\phi^{(s)}:\Omega \to \Rbb$ is a sigmoid feedforward network, $\phi^{(e)}:\Omega \to \Rbb^d$ is a ReLU feedforward network, and $\phi^{(r)}:\Omega \to \Rbb^d$ is a ReLU ResNet, each of depth $L$ and constant width $w$ as shown in Table \ref{tab:parameters}. 
We decided to use the structure above in order to represent both linear functions through $\phi^{(r)}$ and exponential functions through $\phi^{(e)}$, while $\phi^{(s)}$ acts as a scaling network to handle the varying velocities the different regions of the parameter space may use. It appears to perform better than the other few generic networks we tested.

The loss function \eqref{eq:many_control_problem} records the total error accumulated along the trajectories. Sometimes it can be further augmented by an additional loss containing relative errors to target values obtained by, e.g., \cite{du2021evolutional}. In \cite{du2021evolutional}, the goal is to solve a given PDE (not to approximate solution operator) but it provides potential target locations $\bar{\theta}(T)$ by solving a sequence of linear least squares problems from a given $\bar{\theta}(0)$. We leverage this method to generate an additional set $\Tcal=\{\bar{\theta}_i(T):\ i=1,\dots,|\Tcal|\}$, and add the following augmentation term to the loss function \eqref{eq:many_control_problem}:
\begin{equation}
\label{eq:total-loss}
    \ell_{aug}(\xi):=\frac{1}{|\Tcal|}\sum_{i=1}^{|\Tcal|}\Big(\ell({\gamma}_i(T);\xi)+\frac{\|u_{{\theta}_i(T)}-u_{\bar{\theta}_i(T)}\|_{2}^2}{\|u_{\bar{\theta}_i(T)}\|_{2}^2}\Big),
\end{equation}
where the first term in the sum is due to the same logic in \eqref{eq:many_control_problem} by starting from $\theta_i(0) = \bar{\theta}_i(0)$ for all $i$ and the second term is to match the targets computed by \cite{du2021evolutional}.
Here ${\theta}_i(T)$ is the first $m$ components of $\gamma_i(T)$.
Since time-marching is very slow, we only add a small sized $\Tcal$ for some of our experiments. Nevertheless, they help to further improve approximation of $V_{\xi}$ in certain experiments.
%
%
Finally, we note that the empirical estimations in the relative errors are on $L^2(\Omega)$ norm and thus are again approximated by Monte Carlo integration over $\Omega$. 
%
%
Other parameters are also shown in Table \ref{tab:parameters}.
%

For all the experiments, we use the standard ADAM optimizer \cite{kingma2015adam:} with learning rate 0.0005, $\beta_1=0.9$, $\beta_2=0.999$.
We terminate the training process when the percent decrease of the loss is less than $0.1 \%$ averaged over the past 20 steps or 10,000 total iterations whichever comes first. 
Once $V_{\xi}$ is learned, we use the adaptive Dormand-Prince (DOPRI) method to approximate $\theta(t)$ from the ODE in \eqref{eq:ode}. All the implementations and experiments are performed using PyTorch in Python 3.9 in Windows 10 OS on a desktop computer with an AMD Ryzen 7 3800X 8-Core Processor at 3.90 GHz, 16 GB of system memory, and an Nvidia GeForce RTX 2080 Super GPU with 8 GB of graphics memory. 

\begin{table}[]
\caption{The problem dimension $d$, the size of augument set $\Tcal$, the mini-batch size $K$ in training, and the width and depth of $V_{\xi}$ used in each of the numerical experiments.}
    \label{tab:parameters}
    \centering
    \begin{tabular}{ccccccc}
    \toprule
       Experiment  & $d$ &  $|\Tcal|$ & $K$  & Width & Depth \\
       \midrule
         Heat Equation &10&0 &100  & 1000 & 5 \\
         Hyperbolic PDE &10&1000 &100  & 600 & 5 \\
         HJB Equation &8&250 &512  & 1000 & 5 \\
         \bottomrule
    \end{tabular}
\end{table}

\subsection{Numerical Results}
\subsubsection{Comparison to Existing Methods}
First, we consider an initial value problem corresponding to the 10D heat equation ($d=10$):
\begin{equation}
\partial_t u(x,t) = \Delta u(x,t), \quad \forall\, x \in \Omega, t \in [0,T]
\label{eq:ivp-heat-eq}
\end{equation}
with initial value $u(x,0)=g(x)$, where $\Omega=(-1,1)^{10} \subset \Rbb^{10}$ and we use periodic boundary conditions.
Note that this is beyond the restrictions we used in error analysis, which suggests that the analysis may be extended to more general cases. 
As most of the initial conditions result in rapid diffusion, we use $T=0.1$ in this test. We define our neural network as follows:
\begin{equation}
\label{eq:experiment-u}
u_{\theta}(x)=\sum_{i=1}^{80} c_i \tanh \left(a_i^{\top}\sin\left(\pi(x-\beta)\right)-b_i\right).
\end{equation}
We choose this network with trainable parameters $\theta = (\beta, a_i,b_i,c_i)$ with $i=1,\dots,80$, $a_i,\beta \in \Rbb^d$ and $b_i,c_i \in \Rbb$ which enforce the periodic boundary conditions as $\sin(\pi(1-p))=\sin(\pi(-1-p))$ for any $p \in \Rbb$ and the $\sin$ function above acts component-wisely.
%
Hence, $\theta \in \Rbb^{m}$ has dimension $m=970$. We train the vector field $V_{\xi}$ using the nonlinear least squares (NLS) method in \cite{gaby2023neural} and the proposed method (Alg.~\ref{alg:neural-control}) and test their performance on randomly generated initial conditions. 
Specifically, we generate a set of 100,000 points sampled uniformly from $\Theta:=\{\theta\in \Rbb^m :\ |\theta|\leq 20\}$ and another 50,000 from $N(0_m,0.5I_m)$, use them as the initials to train the control field $V_{\xi}$ defined in \eqref{eq:V-network-def} for both of the compared methods. 
%
%
%
%
\begin{figure}
    \begin{subfigure}{0.45\textwidth}
    \includegraphics[width=\linewidth]{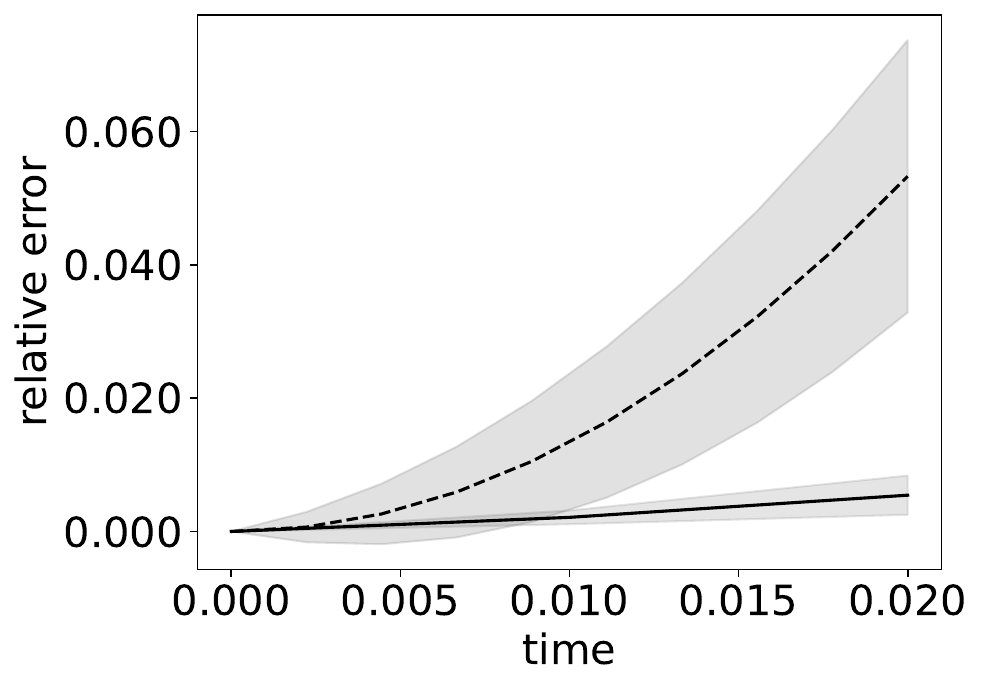}
    \caption{$T=0.02$}
    \label{fig:errors-a}
    \end{subfigure}
    \begin{subfigure}{0.45\textwidth}
    \includegraphics[width=\linewidth]{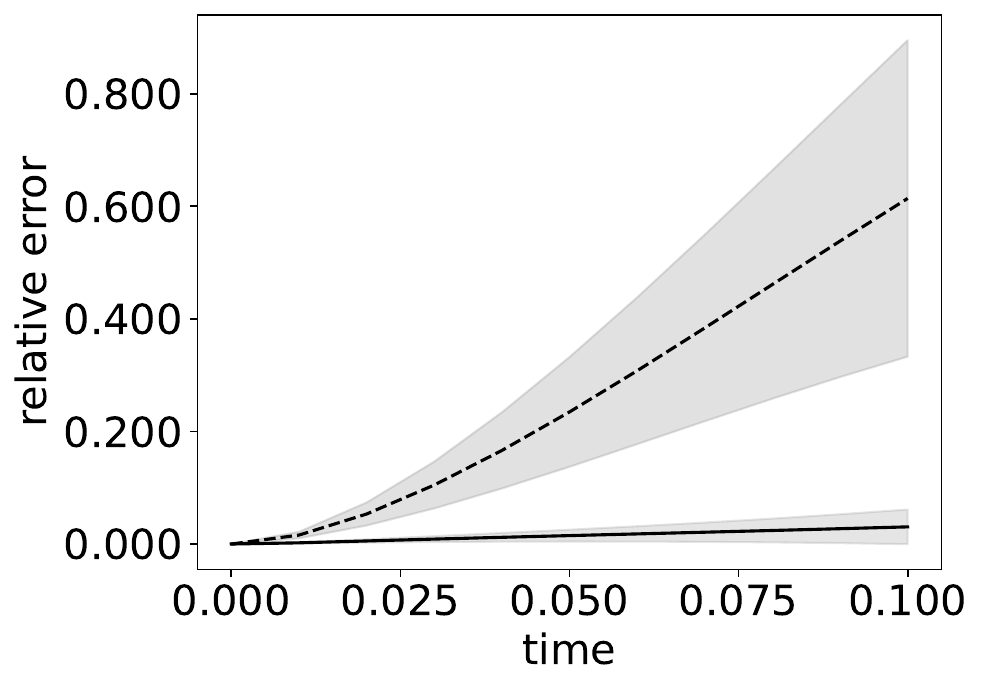}
    \caption{$T=0.10$}
    \label{fig:errors-b}
    \end{subfigure}
    \caption{Comparison of NLS \cite{gaby2023neural} method (black dashed line) and the proposed method (black solid line). While the previous method only allows accurate results for a small time scale $T=0.02$, the new method allows more accuracy for longer time scales $T=0.1$. 
    }
    \label{fig:error-heat}
\end{figure}

After the training is completed, we generate 100 random initials $g=u_{\theta(0)}$ where $\theta(0) \sim N(0_m,0.5I_m)$. For each $g$, we compute the ODE \eqref{eq:ode} using the trained $V_{\xi}$ of the comparison methods, obtain an approximate solution $u_{\theta(t)}$, and compute the relative error $\|\utt (\cdot) - u^*(\cdot,t)\|_2^2/\|u^*(\cdot,t)\|_2^2$ where the true solution is given by $u^*(x,t)=\int g(y)N(y-x,2t I_d)\,dy$ \cite{evans1998partial}, where $N$ stands for the density of Gaussian.
Then we show the average and standard deviation of the relative errors over these 100 initials versus time $t$ in  Figure \ref{fig:error-heat}. 

Figure \ref{fig:error-heat} shows that Alg.~\ref{alg:neural-control} significantly improves the approximation accuracy compared to the NLS method provided by \cite{gaby2023neural}. Specifically, $V_{\xi}$ obtained by NLS yields reasonably accurate solutions up to $t= 0.01$ with $<2\%$ relative error, but the error grows unacceptably large to $60\%$ on average at $t=0.1$. 
On the other hand, the vector field $V_{\xi}$ trained by Alg.~\ref{alg:neural-control} yields much more stable and accurate solution $u_{\theta(t)}$. The relative error maintains to be lower than $0.3\%$ at $t=0.01$ and $4\%$ for the proposed method at $t=0.1$.
We also show the approximation $u_{\theta(t)}$ obtained by NLS and the proposed method for $t=0.02$ in Figure \ref{fig:heat_qualitative}. Due to space limitation, we only draw 5 different initial function $g$ (as shown from the top to bottom rows in Figure \ref{fig:heat_qualitative}). From left to right columns, they are the initial $g$ (first column), the true solution $u^*$ (second column), the result obtained by the proposed method (third column), and the approximation obtained by NLS (fourth column), where the last three are at $t=0.02$.

The training time for NLS takes approximately 1.3 hours using NLS, but only about 15 minutes using Alg.~\ref{alg:neural-control}. 
The testing time for individual initial condition, i.e., running time of Alg.~\ref{alg:solving-ivp}, takes an average of 0.84 second with a standard deviation of 0.33 second. We remark that this time can be significantly reduced when multiple initials are queried in batch. For example, it takes only 1.94 seconds to solve 100 different initials when they are batched. 
%
%
These results demonstrate the significant improvement of the proposed method in both accuracy and efficiency. 
%
\begin{figure}
    \centering
\includegraphics[width=0.95\linewidth]{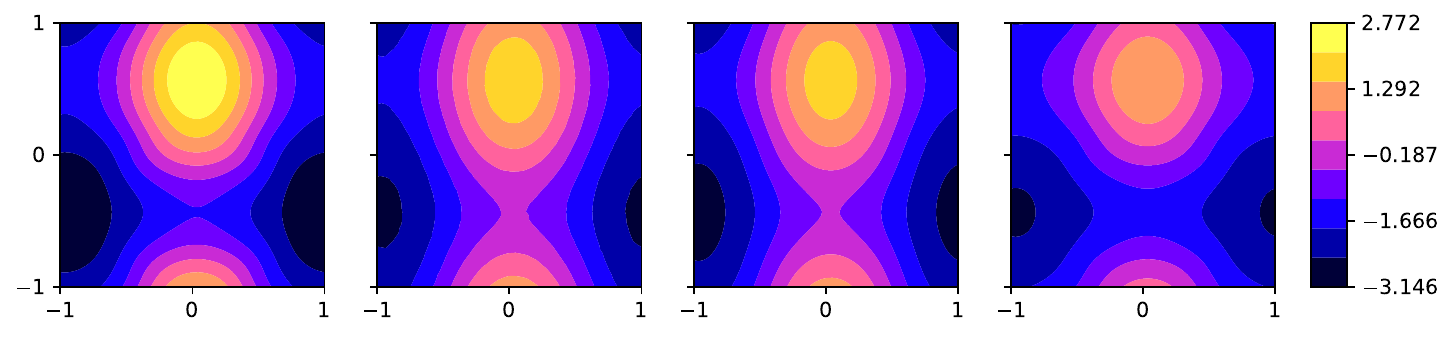}
\includegraphics[width=0.95\linewidth]{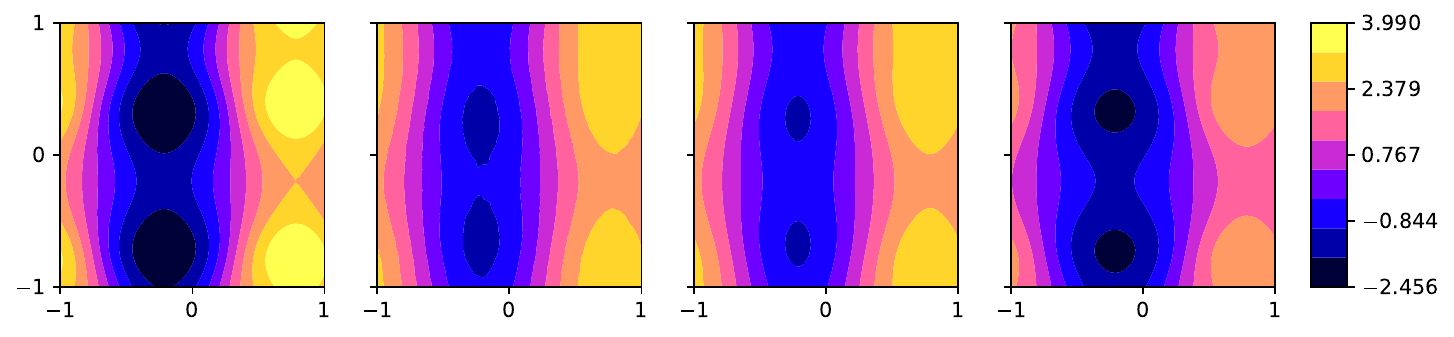}
\includegraphics[width=0.95\linewidth]{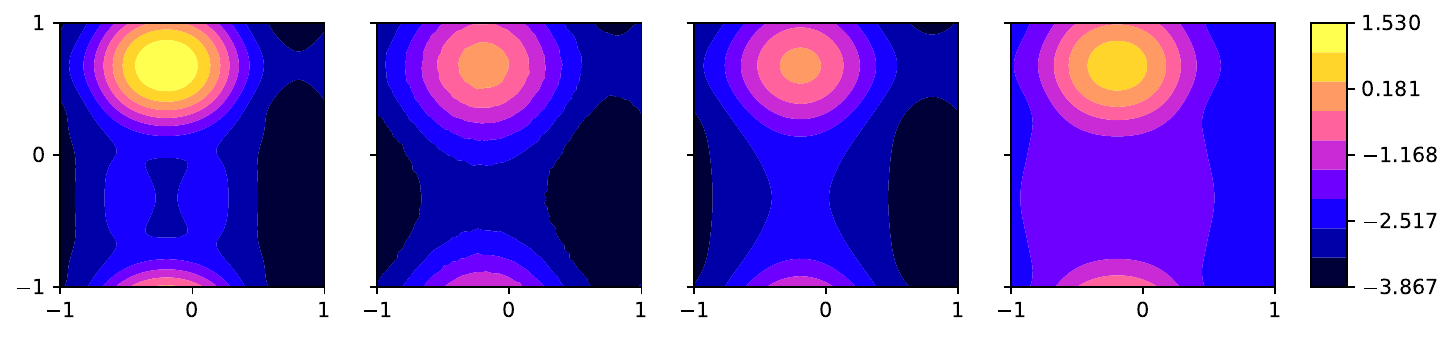}
\includegraphics[width=0.95\linewidth]{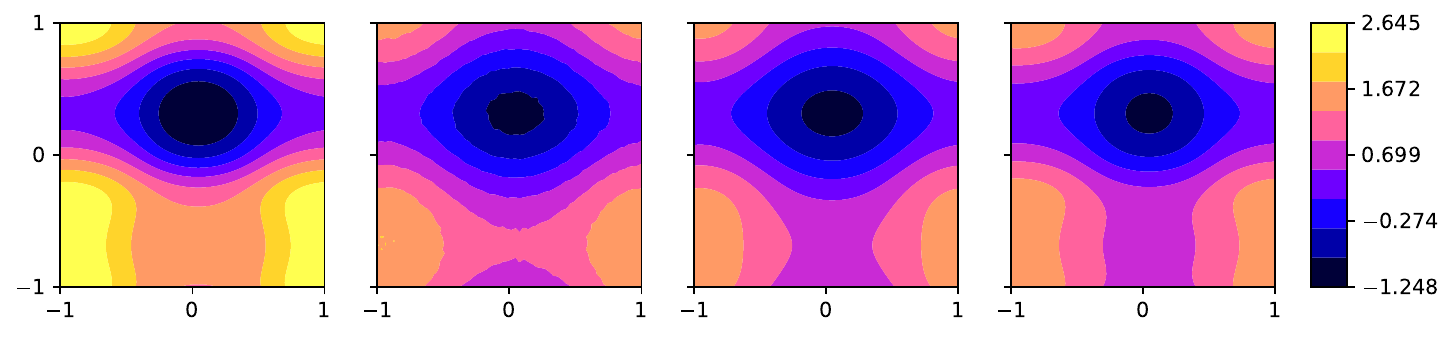}
\includegraphics[width=0.95\linewidth]{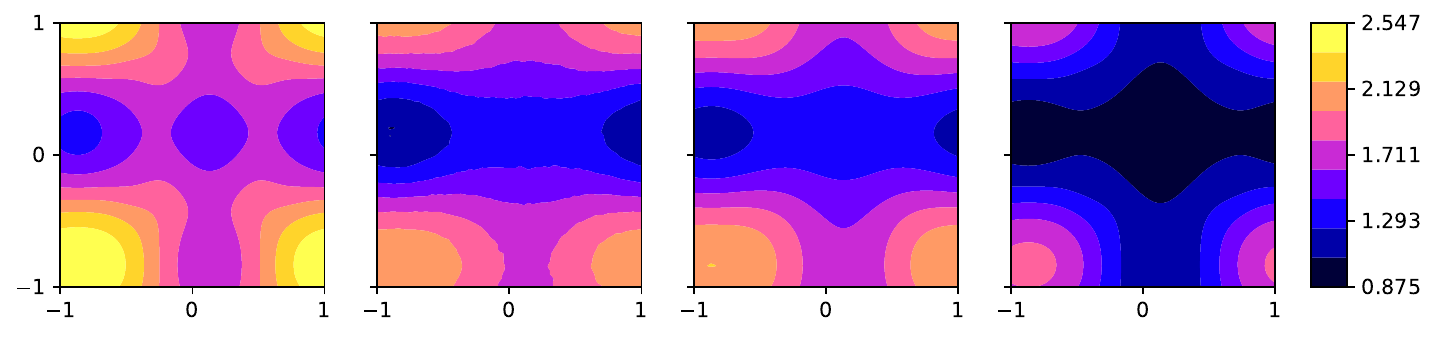}
    \caption{Comparison with NLS \cite{gaby2023neural} on the Heat Equation \eqref{eq:ivp-heat-eq}. All plots show the $(x_1,x_2)$ plane. (First column) initial $g$; (Second column): reference solution; (Third column) Proposed method; and (Fourth column) NLS \cite{gaby2023neural}. In all cases, the proposed method demonstrates significant improvement on solution accuracy. 
    }
    \label{fig:heat_qualitative}
\end{figure}

\subsubsection{Hyperbolic PDE}
Our next test is on the following 10D nonlinear hyperbolic PDE:
\begin{equation}
\label{eq:hyperbolic}    
\partial_t u(x,t) = F[u](x,t):=2\nabla \cdot \tanh(u(x,t))
\end{equation}
with initial value $u(x,0)=g(x)$ and $\Omega = (-1,1)^{10}$.
We remark that this is a very challenging problem and its solutions can often develop shocks which makes most methods suffer.
Therefore, we choose $T=0.15$ as this is the earliest we observed shock waves in the examples we shall compare against. 
Then we enforce periodic boundary conditions by using the same network $\ut$ in \eqref{eq:experiment-u} as above. Here, we sample $\theta_i(0)$ such that $a_i,\beta \sim N(0_d,I_d)$ and $b_i \sim N(0,1)$ for $i =1,\ldots, 80$, and $c=(c_1,c_2,\ldots,c_{80})^{\top}$ is sampled uniformly from the sphere of radius 1. For this method to be scalable, we have to use special initials $g=u_{\theta(0)}$ with $a_{i}(0)=(a_{i1}(0),0,\ldots,0)^{\top}$ for $i=1,\ldots,80$. For the $\theta$'s to be used in the terminal set $\Tcal$, we make use of this form of our tested $g$'s and sample $\beta,b_i,$ and $c$ as above but use $a_{i}(0)=(a_{i1}(0),0,\ldots,0)^{\top}$ with $a_{i1}\sim N(0,1)$.

The PDE \eqref{eq:hyperbolic} does not have analytic solutions. For comparison purpose, we use a dedicated solver \cite{mao2020physics}, which is a numerical approximation using a first-order upwind scheme, with 4000 time steps and 1000 spatial steps to compute the reference solution. We use as initial conditions the $g$'s defined above as their 1-d nature allows the computation of our reference solutions.
We stress that the initials we use here for testing are not contained within the samples used for training.
After training, we use the learned control $V_{\xi}$ to generate solutions for all these initials.
We plot mean and relative error of our method compared to the reference numerical solutions for 100 different random initials in Figure \ref{fig:error-hyper} (left). We observe a relative error of less than $4\%$ on average, and more than $70\%$ of results are below $8\%$ relative error in our tests.
These results suggest that our proposed method has great potential to be applied to more challenging PDEs. 
%

\begin{figure}
    \centering
    \includegraphics[scale=0.45]{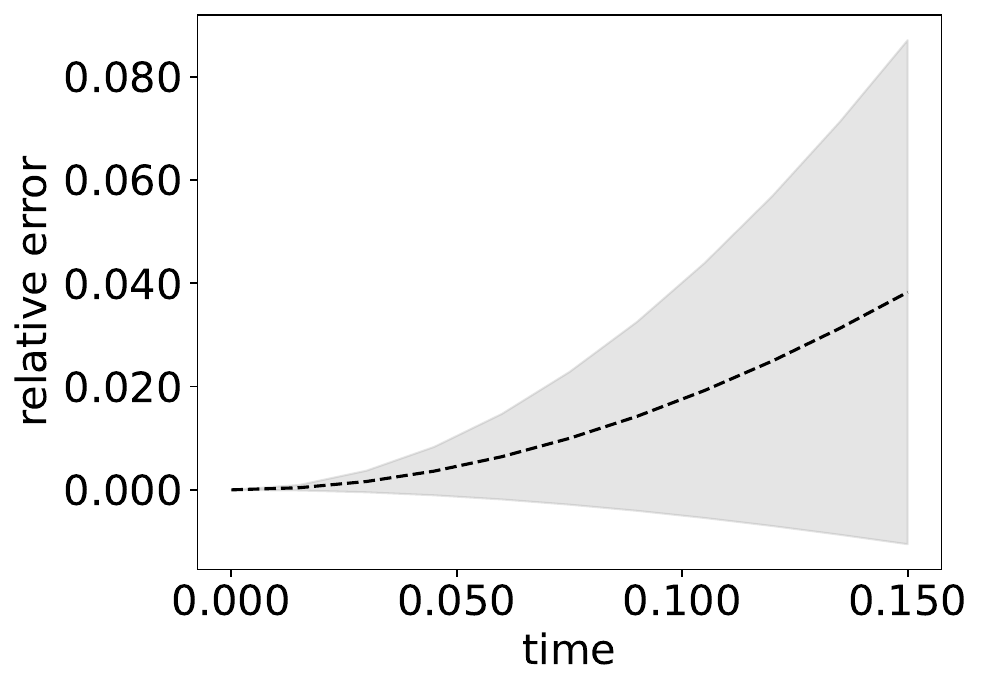}
    \includegraphics[scale=0.45]{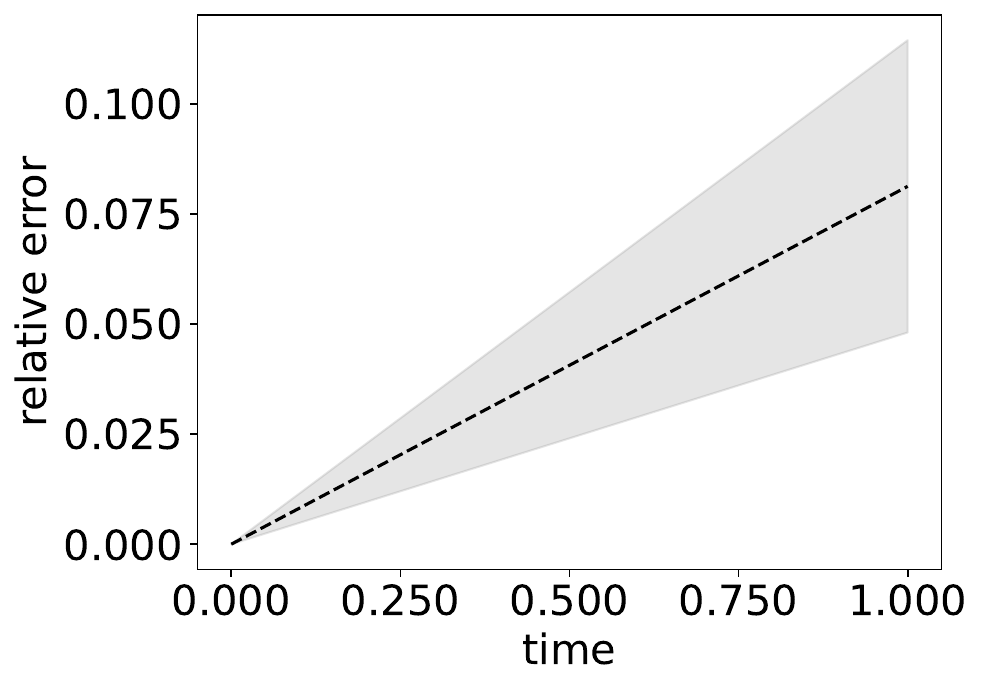}
    \caption{Mean relative error (dotted line) and standard deviation (grey) between the learned $\utt$ and the reference solution $u^*$ for 100 different random initial conditions. (Left) hyperbolic PDE \eqref{eq:hyperbolic}; (Right) Hamiltonian-Jacobi-Bellman equation \eqref{eq:ivp-hjb-eq}. Note that in either case, the tested initials are not included in the training datasets.}
    \label{fig:error-hyper}
\end{figure}

\subsubsection{Applications to Hamilton-Jacobi-Bellman equations}
We consider the following stochastic control problem with control function $\alpha$:
\begin{subequations}
\begin{align}
    \min_{\alpha}\quad &\frac{1}{2}\int_0^{T}|\alpha(X(t),t)|^2 dt+g(X(T)) , \label{eq:sc_obj}\\
    \text{ s.t.}\quad  &dX(t)=\alpha(X(t),t)dt+\sqrt{2\epsilon} dW, \qquad \text{where}\quad X(0)=x, \label{eq:sc_sde}
\end{align}    
\end{subequations}
and the objective function consists of a running cost of $\alpha$ and a terminal cost $g$, and $W$ is the standard Wiener process. 
We set problem dimension to $d=8$, coefficient $\epsilon=0.2$, and terminal time $T=1$.
Here $X(t) \in \Rbb^{d}$ for all $t$.
The corresponding Hamilton-Jacobi-Bellman (HJB) equation of this control problem is a second-order nonlinear PDE of the value function $u$:
\begin{equation}
\partial_t u(x,t) = -\epsilon\Delta u(x,t)+\frac{1}{2}|\nabla u(x,t)|^2, \quad \forall\, x \in \Rbb^d,\ t \in [0,T]
\label{eq:ivp-hjb-eq}
\end{equation}
with terminal cost (value) $u(x,T)=g(x)$.
The optimal control $\alpha(x,t) = -\nabla u(x,t)$ for all $x$ and $t$.
This problem is a classic problem with many real-world applications \cite{fleming1975deterministic}.
Our goal is to find the solution operator that maps any terminal condition $g$ to its corresponding solution. 
%
We note that \eqref{eq:ivp-hjb-eq} can be easily converted to an IVP with initial value $g$ by changing $t$ to $T-t$. We choose
\[
u_{\theta}(x)=\sum_{i=1}^{50} w_i e^{-|a_i \odot (x-b_i)|^2/2},
\] 
where $\odot$ denotes component wise multiplication, and the parameters $\theta \in \Rbb^{m}$ with $m = 850$ is the collection of $(w_i,a_i,b_i) \in \Rbb \times \Rbb^{d} \times \Rbb^{d}$ for $i=1,\dots,50$. We select this neural network as it is similar to those in the set $\Gcal$ below (but they are not necessarily the actual solutions $u^*$) and satisfies $|u_{\theta}(x)|\to 0$ as $x \to \infty$ so long as $a_i \neq 0$. This is important as our solutions will be mainly concentrated in $[-3,3]^d$ due to the solution property of the HJB equation and we want our neural network to reflect this.
We sample $g$ functions by setting them to $u_{\theta(0)}$ for $\theta(0)$ sampled from 
\[
\Theta:=\{\theta :\ w_i \in (-1,0),\ |a_i|_{\infty} \in (0.1,2),\ |b_i|_{\infty} \in [-2,2] \},
\]
and use these to generate trajectories $\theta(t)$ for training. We sampled 250 points $\theta_i(0)'s$ from $\Theta$ uniformly and then used the time marching numerical method as described in \cite{du2021evolutional} to generate a terminal dataset $\Tcal$. As such we use the loss function \eqref{eq:total-loss} in our training. As $\Omega = \Rbb^8$ we use importance sampling to generate the $x$ points for our Monte-Carlo approximations of $\| \cdot \|_2$. Namely, we sample $x$ from the distribution defined by the density $\rho(x; \theta):=\frac{1}{50}\sum_{i=1}^{50} N(x-b_i, \text{diag}(a_i)^{-2})$.

\begin{figure}[h]
    \centering
    \includegraphics[scale=0.65]{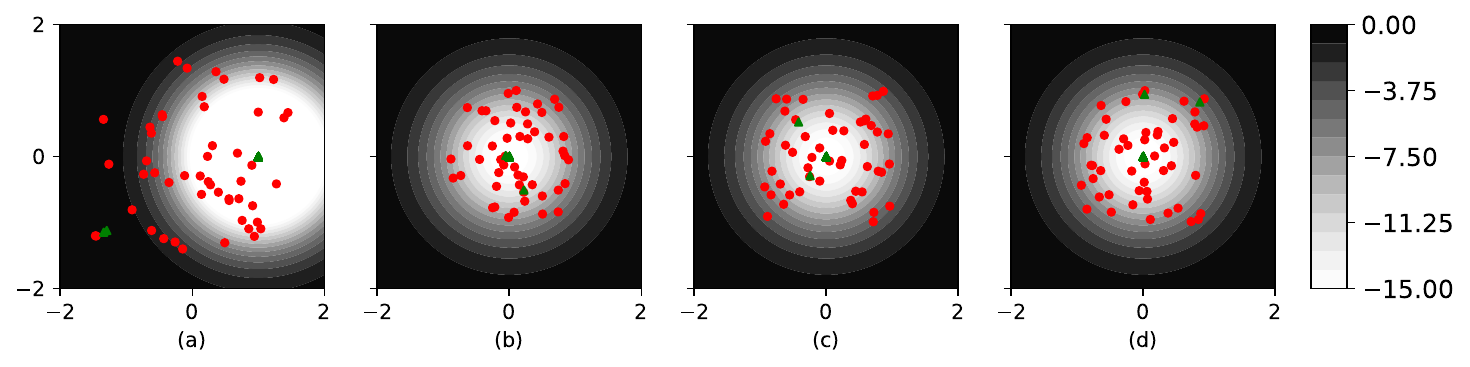}\vspace{-12pt}
     \includegraphics[scale=0.65]{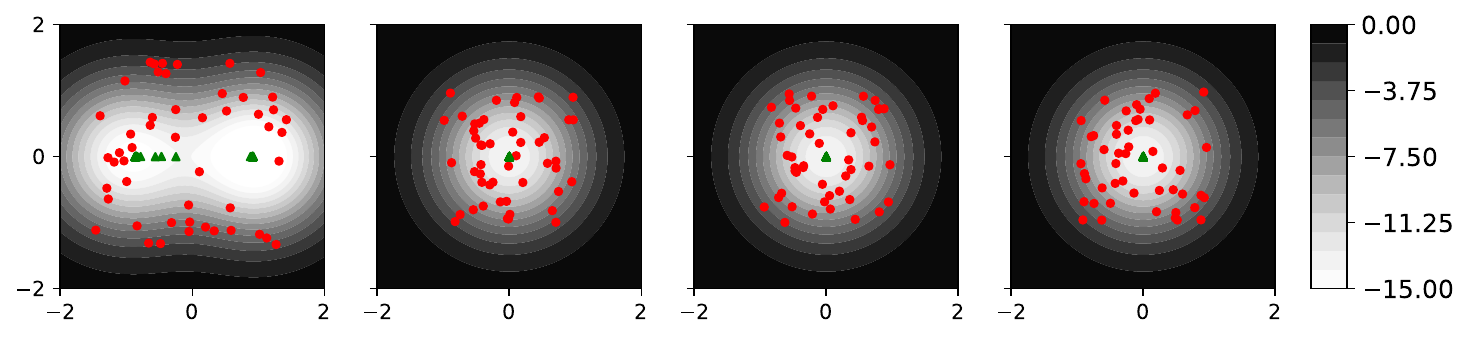}
     \includegraphics[scale=0.65]{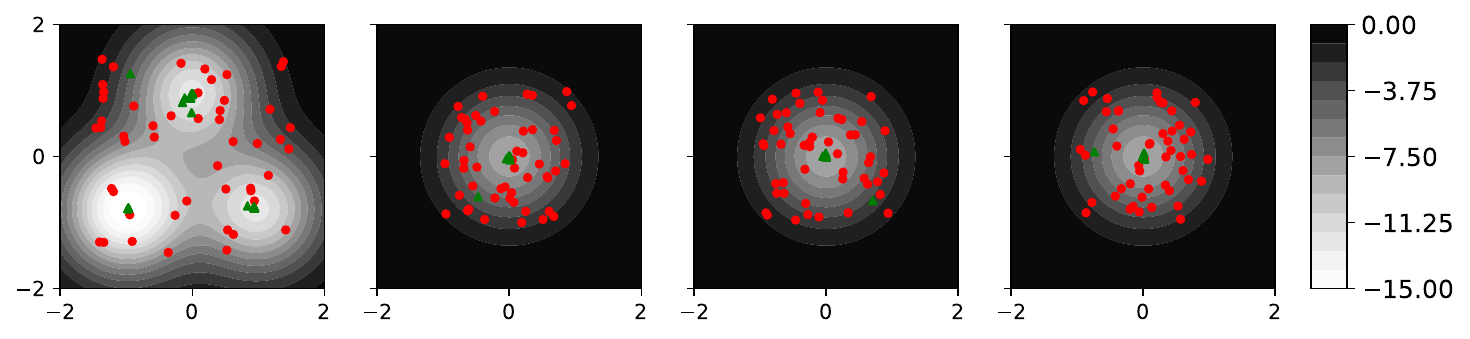}
     \includegraphics[scale=0.65]{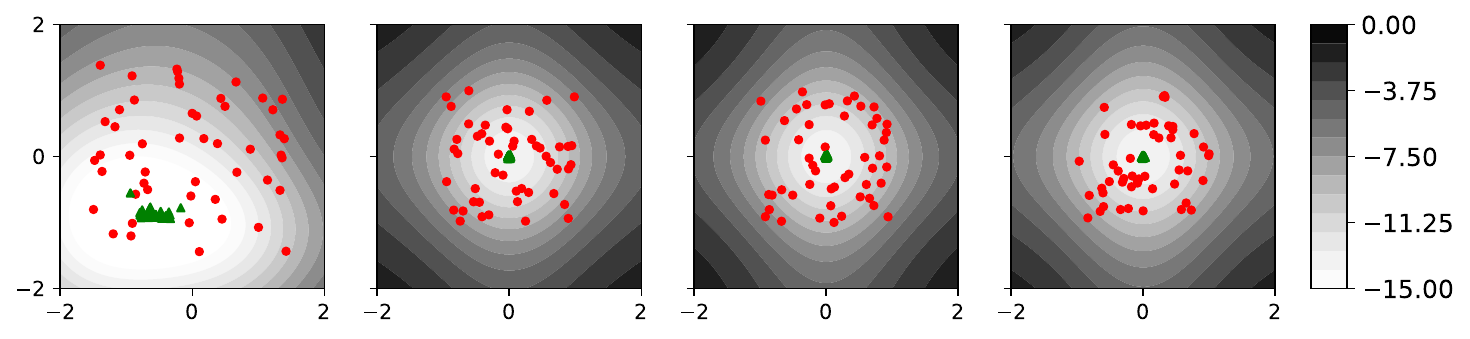}
     \includegraphics[scale=0.65]{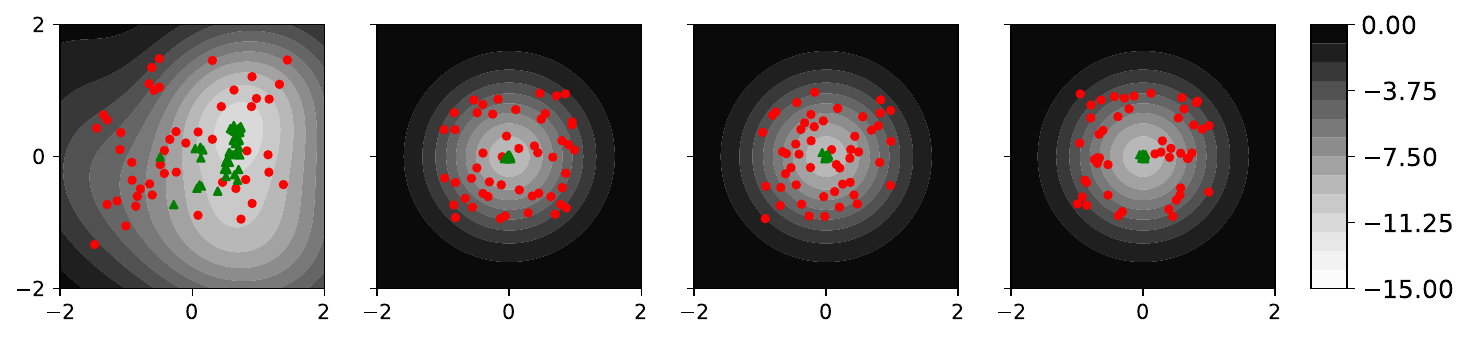}
    \caption{The evolution of 50 sampled points $X(0)$ (red circles) to time $X(1)$ (green triangles) in the (first column) $(x_1,x_2)$ plane; (second column) $(x_3,x_4)$ plane; (third column) $(x_5,x_6)$ plane; and (fourth column) $(x_7,x_8)$ plane. The background images show the expected minimum points of the terminal cost for the five randomly chosen initials. We see the solution $-\nabla u_{\theta(t)}$ provides correct control for all cases. Note that the control may not be able to steer very far initials (those started from the black regions) to the minimum (bright regions) since excessive movement causes large running costs that are not compensated by the terminal gains. As the terminal cost $g$ is scaled larger, this phenomenon becomes less likely to happen, and vice versa.
    }
    \label{fig:hjb-points}
\end{figure}

After we train the $V_{\xi}$, we uniformly sample 100 terminal conditions $g$ from the following set:
\[
\Gcal:=\Big\{\sum_{i=1}^{50} c_i e^{-|x-b_i|^2/\sigma_i^2} \ : \ c_i \in (-1,0),\ \sigma_i^2 \in (0.5,20),\ |b_i|_{\infty} \in [-2,2]\Big\}.
\]
We choose such $\Gcal$ as it is an analogue to the Reproducing Kernel Hilbert Space (RKHS) and can represent a variety of functions by different combinations of functions in practice.
%
For each $g$, we compute the solution using the Cole-Hopf transformation \cite{evans1998partial}:
\[
u^*(x,t)=-2\epsilon\ln{\left((4\epsilon\pi (T-t))^{-d/2}\int_{\Rbb^d}\text{exp}\left(-\frac{|x-y|^2}{4\epsilon (T-t)}-\frac{g(y)}{2\epsilon(T-t)}\right)dy\right)}.
\]
We can approximate these true solutions using Monte-Carlo methods, which is done by using importance sampling to sample 20,000 points in $\Rbb^8$ using the distribution $N(x,2\epsilon(T-t))$. This allows us
to approximate the integral above and compute $u^*$ at each $(x,t)$.
%
%
We graph the average and standard deviation of the relative error of the solution by our approach for these 100 terminal costs in Figure \ref{fig:error-hyper} (right) with reversed time $t \leftarrow T-t$. This plot shows that the average relative error is about $7.5\%$ and the standard deviation is less than $3\%$. 
%
%

%
For demonstration, we show the performance of control $\hat{\alpha}(\cdot,t) := -\nabla \utt(\cdot)$ obtained above. Due to space limitation, we select 5 terminal cost $g$'s such that they vary in $x_1,x_2$ making the minimum locations clear.
%
For each $g$, we sample 50 $X_0$'s uniformly from $[-1.5,1.5]^2\times[-1,1]^6$ and generate their trajectories by solving the stochastic differential equation with $\hat{\alpha}$ in \eqref{eq:sc_sde} using the Euler-Maruyama method with step size 0.001.
The initial $X(0)$'s (red dots) and the terminal locations $X(T)$'s (green triangle) averaged over 1000 runs for these terminals are shown from top to bottom rows in Figure \ref{fig:hjb-points}. The gray-level contours show the corresponding terminal cost $g$ and the locations of $X$ are shown in the $(x_1,x_2)$, $(x_3,x_4)$, $(x_5,x_6)$, and $(x_7,x_8)$ planes in the four columns of Figure \ref{fig:hjb-points}. 
%
%
%
From these plots, one can see properly controlled particle locations by using the solution obtained by the proposed method.

\section{Conclusion and Future Work}
\label{sec:conclusion}

We proposed a novel control-based method to approximate solution operators of evolution PDEs, particularly in the high-dimension. Specifically, we use a reduced-order model (such as a deep neural network) to represent the solutions of a PDE, and learn the control field in the space of model parameters. The trajectories under the learned control render the model with time-evolving parameters to accurately approximate the true solution. The proposed method demonstrate substantially improved solution accuracy and efficiency over all existing works, and it allows users to quickly and accurately approximate the solutions of high-dimensional evolution PDEs with arbitrary initial values by solving controlling ODEs of the parameters. Rigorous approximation estimates are provided for a general class of second-order semilinear PDEs. The unique methodology discussed in this paper provides a unique and exciting direction to approximate solution operators of high-dimensional PDEs.

Despite the promising results obtained in this work, there are still many open problems remain to be solved in this direction. For example, we observed further improvements on parameter efficiency and solution quality when using architectures, particularly $u_{\theta}$, inspired by potential solutions of the PDE. In practice, is it possibly to build such architectures when the solution structure (not the solutions themselves) is only completely or only partially known? Is there a more adaptive approach when the PDE solution may render oscillations or shocks? Could the error analysis be extended to other types of nonlinear PDEs to ensure solution accuracy for a broader range of applications? These problems are very important to address along the direction of the proposed solution methodology.

\appendix

\section{Proof of Theorem \ref{thm:theta_exist}}
\label{appx:A}
We proceed by showing the following Lemma, which provides a meaningful connection between commonly-used general neural networks with fixed parameters and neural networks with time-evolving parameters. For ease of presentation, we restrict to standard feed-forward networks only, while the idea can be readily generalized to other networks  (see Remark \ref{rmk:static-evolving-nn}).
\begin{lemma}
\label{lem:evolution_params}
For any feed-forward neural network architecture $v_{\eta}(x,t)$ with parameters denoted by $\eta$, there exists a feed-forward neural network $\utt (x)$ and a differentiable evolution of the parameters $\theta(t)$ such that $\utt(x) = v_{\eta}(x,t)$ for all $x \in \Rbb^{d}$ and $t \in \Rbb$.
\end{lemma}
\begin{proof}
    We denote the feed-forward network $v_{\eta}(x,t)=w^{\top}_l h_{l-1}(x,t) + b_l$ with
    \[
    h_{\ell}(x,t) = \sigma(W_{\ell}h_{\ell-1}(x,t)+b_{\ell}) , \quad \ell = 1,\ \ldots,\ l-1,
    \] 
    where $h_0(x,t)= \sigma(W_0x+w_0t+b_0)$ and $\eta = (w_l, b_l, W_{l-1},\dots, W_0, w_0, b_0)$. Note that all components in $\eta$ are constants independent of $x$ and $t$. Similarly, denote $u_{\theta}(x)=\tilde{w}_l^{\top}\tilde{h}_{l-1}(x)+ \tilde{b}_l$ with
    \[
    \tilde{h}_{\ell}(x) = \sigma(\tilde{W}_{\ell}\tilde{h}_{\ell-1}(x)+\tilde{b}_{\ell}), \quad \ell = 1,\ \ldots,\ l-1,
    \]
    where $\tilde{h}_0(x)=\sigma(\tilde{W}_0x + \tilde{b}_0)$ and $\theta=(\tilde{w}_l,\tilde{b}_l,\tilde{W}_{l-1},\dots,\tilde{W}_0,\tilde{b}_{0})$. Note that $\theta(t)$ can be time-evolving. Hence, we can set $\tilde{b}_{0}(t) = w_0 t+b_0$ for all $t$ and all other components in $\theta(t)$ as constants identical to the corresponding ones in $\eta$, i.e., $\tilde{w}_l(t)\equiv w_l$, $\tilde{b}_L(t)\equiv b_L$, $\dots$, $\tilde{b}_1(t) \equiv 1_0$. Then it is clear that $\utt(x) = v_{\eta}(x,t)$ for all $x$ and $t$, and $\theta(t)$ is differentiable. 
\end{proof}
\begin{remark}
\label{rmk:static-evolving-nn}
    While we state Lemma \ref{lem:evolution_params} for feed-forward neural networks, it is easy to see that the result holds for general neural networks as well. To see this, note that for any neuron $y(x,t)$ we can rewrite it as $\tilde{y}_{\eta(t)}(x)$ as shown in Lemma \ref{lem:evolution_params}. For any neuron $z(t,x,y(x,t))$ we can write it as $z(t,x,y(x,t))= z(t,x,\tilde{y}_{\eta(t)}(x)) = \tilde{z}_{(\eta(t),\tilde{\eta}(t))}(x)$ for some neural network $\tilde{z}_{\eta, \tilde{\eta}}$ with input $x$ and parameters $(\eta,\tilde{\eta})$ evolving in $t$. Applying this repeatedly justifies the claim.
\end{remark}
We will use the following statement of the universal approximation theorem for neural networks in our later proofs. 
\begin{lemma}[Theorem 5.1 \cite{deryck2021approximation}]
\label{lem:universal}
    Let $\epsilon>0$, $L>0$, and $k\in \mathbb{N}$. There exists a fixed feed-forward neural network architecture $w_{\eta}$ with $\tanh$ activation such that for all $f\in SB(\Omega,L,k+1)$ there exists an $\eta \in \Rbb^m$ such that
    \[
        \|w_{\eta}-f\|_{W^{k,\infty}(\Omega)}<\epsilon.
    \]
\end{lemma}
\begin{proof}
    While Theorem 5.1 in \cite{deryck2021approximation} is stated for $\Omega=[0,1]^d$ we can apply a transformation layer to generalize this for more arbitrary bounded $\Omega$. Thus Lemma \ref{lem:universal} is an immediate corollary to Theorem 5.1 in \cite{deryck2021approximation}.
\end{proof}
We now advance to the main proof of this subsection.
\begin{proof}[Proof of Theroem \ref{thm:theta_exist}]
    Let $\delta>0$ and define $\Omega_{\delta}=\Omega \times [-\delta, \delta]$. Denote
    \[
    \bar{L}:= (1 + (1+\delta) M_{k,L,F}) L + (1+\delta) \|F[0]\|_{W^{k+1,\infty}(\Omega)}.
    \]
    For any $g \in W^{2k+3,\infty}(\Omega)$ with $\|g\|_{W^{2k+2,\infty}(\Omega)} \le L$, we define $T_g$ such that $T_g(x,t) = g(x) + t F[g](x)$ for all $(x,t) \in \Omega_{\delta}$. Note that $T_g(\cdot,0) \equiv g(\cdot)$ and $\partial_t T_g(\cdot,t)=F[g](\cdot)$ for all $t$. Furthhermore, there is
    \begin{align}
        \|T_g\|_{W^{k+2,\infty}(\Omega_{\delta})}
        & = \|g + t F[g] \|_{W^{k+2,\infty}(\Omega_{\delta})} \nonumber \\
        & \le \|g\|_{W^{k+2,\infty}(\Omega_{\delta})} + \|F[g]\|_{W^{k+1,\infty}(\Omega)} + \delta \|F[g]\|_{W^{k+2,\infty}(\Omega)} \label{eq:Tg_decompose} \\
        & \le \|g\|_{W^{k+2,\infty}(\Omega_{\delta})} + (1+\delta)\|F[g]\|_{W^{k+2,\infty}(\Omega)}\nonumber \\
        & =: \text{I} + (1+\delta)\cdot \text{II}, \nonumber
    \end{align}
    where the first equality in \eqref{eq:Tg_decompose} is due to the definition of $T_g$, the first inequality due to the triangle inequality in $W^{k+2,\infty}(\Omega_{\delta})$, the split of time and spatial derivatives on $T_g$, and $|t|\le \delta$, and the second inequality due to $\|F[g]\|_{W^{k+1,\infty}(\Omega_{\delta})} \le \|F[g]\|_{W^{k+2,\infty}(\Omega_{\delta})}$. Since $g \in C^{2k+2}(\Omega)$, we have
    \begin{equation}
    \label{eq:lemma_g_bound}
        \text{I}:= \|g\|_{W^{k+2,\infty}(\Omega_{\delta})} \le \|g\|_{W^{2k+3,\infty}(\Omega)} \le L
    \end{equation}
    given that $g$ is constant in $t$.
    On the other hand, there is
    \begin{align}
        \text{II}
        := \|F[g]\|_{W^{k+2,\infty}(\Omega)} 
        & \le \|F[g] - F[0]\|_{W^{k+2,\infty}(\Omega)} + \|F[0]\|_{W^{k+2,\infty}(\Omega)} \nonumber\\
        & \le M_{k,L,F} \|g\|_{W^{2k+2,\infty}(\Omega)} + \|F[0]\|_{W^{k+2,\infty}(\Omega)} \label{eq:lemma_Fg_bound}  \\
        & \le M_{k,L,F} L + \|F[0]\|_{W^{k+2,\infty}(\Omega)} \nonumber
    \end{align}
    where the second inequality is due to the Assumption \ref{assump:regularity} on $F$.
    Combining \eqref{eq:Tg_decompose}, \eqref{eq:lemma_g_bound} and \eqref{eq:lemma_Fg_bound} yields $\|T_g\|_{W^{k+2,\infty}(\Omega_{\delta})} \le \bar{L}$.
    By Lemma \ref{lem:universal}, there exists a feedforward neural network architecture $v_{\eta}$ such that for any $T_g$ there are parameters $\eta = \eta_g$ for some $\eta_g$ satisfying
    \begin{equation}
    \label{eq:v_approx_Tg}
        \| v_{\eta_g} - T_g \|_{W^{k+1,\infty}(\Omega_{\delta})} < \frac{\epsilon}{2(1+M_{k,L,F})}.
    \end{equation}

    Define $w_g := v_{\eta_g} - T_g$, then $w_g \in W^{k+2,\infty}(\Omega_{\delta})$ since $v_{\eta_g}$ is smooth in $(x,t)$ and $T_g(\cdot,t)=g(\cdot) + t F[g](\cdot) \in W^{k+2,\infty}(\Omega_{\delta})$. As a result, for any multi-index $\alpha \in \Nbb^{d+1}$ satisfying $|\alpha| \le k+1 $, there is $\partial^{\alpha}w_g \in W^{1,\infty}(\Omega_{\delta})$, which also implies $\partial^{\alpha} w_g \in C(\Omega_{\delta})$. Due to this continuity, we also have
    \begin{equation}
        \| \partial^{\alpha} w_g(\cdot, 0)\|_{L^{\infty}(\Omega)} \le \| \partial^{\alpha} w_g \|_{L^{\infty}(\Omega_{\delta})}.
    \end{equation}
    Since $\alpha$ is arbitrary, we thus have
    \begin{align}
        \| w_g(\cdot, 0)\|_{W^{p,\infty}(\Omega)} & \le \| w_g \|_{W^{p,\infty}(\Omega_{\delta})} \label{eq:w0_bound} \\ 
        \| \partial_t w_g(\cdot, 0)\|_{W^{p-1,\infty}(\Omega)} & \le \| \partial_t w_g \|_{W^{p-1,\infty}(\Omega_{\delta})} \label{eq:dtw0_bound}
    \end{align}
    for all $p=1,\dots, k+1$. Therefore,we have 
    \begin{align}
         \|v_{\eta_g}(\cdot,0)-g(\cdot)\|_{W^{k+1,\infty}(\Omega)}
         & = \|w_{g}(\cdot,0)\|_{W^{k+1,\infty}(\Omega)} 
         \le \|w_{g}\|_{W^{k+1,\infty}(\Omega_{\delta})} \label{eq:wg_bound} \\
         & = \|v_{\eta_g}-T_g\|_{W^{k+1,\infty}(\Omega_{\delta})}<\frac{\epsilon}{2(1+M_{k,L,F})}, \nonumber
    \end{align}
    where the two equalities are due to the definition of $w_g$, the first inequality due to  \eqref{eq:w0_bound} with $p=k+1$, and last inequality due to \eqref{eq:v_approx_Tg}, 
    Note that \eqref{eq:wg_bound} further implies that
    \begin{equation}
    \label{eq:F_vg_g_bound}
        \|F[v_{\eta_g}](\cdot,0) - F[g](\cdot)\|_{W^{1,\infty}(\Omega)} \le M_{k,L,F} \| v_{\eta_g}(\cdot,0) - g(\cdot) \|_{W^{k+1,\infty}(\Omega)} < \frac{\epsilon}{2}.
    \end{equation}
    Combining the results above, we have
    \begin{align}
    \|\partial_t v_{\eta_g}(\cdot,0)-F[v_{\eta_g}](\cdot,0)\|_{W^{1,\infty}(\Omega)}& \leq \|\partial_t v_{\eta_g}(\cdot,0)-F[g](\cdot)\|_{W^{1,\infty}(\Omega)} +\|F[g](\cdot)-F[v_{\eta_g}](\cdot,0)\|_{W^{1,\infty}(\Omega)} \nonumber \\
    & < \|\partial_t w_g(\cdot,0)\|_{W^{1,\infty}(\Omega)} +\frac{\epsilon}{2}  \nonumber\\
    & \le \|\partial_t w_g\|_{W^{1,\infty}(\Omega_{\delta})} +\frac{\epsilon}{2}  \label{eq:projection_proof}\\
    & \le \|v_{\eta_g}-T_g\|_{W^{k+1,\infty}(\Omega_{\delta})} +\frac{\epsilon}{2}  \nonumber\\
    &\leq \epsilon. \nonumber
    \end{align}
    where the first inequality above is due to the triangle inequality, the second due to the definition of $w_{g}$, the property $\partial_t T_g(\cdot,0) = F[g](\cdot)$, and \eqref{eq:F_vg_g_bound}, the third due to \eqref{eq:dtw0_bound} with $p=2$, the fourth due to $\|v_{\eta_g}-T_g\|_{W^{k+1,\infty}(\Omega_{\delta})}$ being the max of terms including $\|\partial_t(v_{\eta_g}-T_g)\|_{W^{1,\infty}(\Omega_{\delta})}$, and the last inequality due to \eqref{eq:wg_bound}.
    Since $v_{\eta}$ uses tanh activations and hence is smooth in $\eta$, we know both \eqref{eq:wg_bound} and \eqref{eq:projection_proof} hold for some open neighborhood $U_g$ of $\eta_{g}$ in $\Rbb^{m}$.

    Following the ideas in the proof of Lemma \ref{lem:evolution_params} we define the sets $\hat{U}_g$ where for each $\eta=(w_l,\dots,w_0,b_0) \in U_g$ we have $\theta=(w_l,\dots,b_0)\in \hat{U}_g$, since $U_g$ is open we can easily conclude that $\hat{U}_g$ is also open. 

    Now we set $\Gamma:=\bigcup_{g \in W^{2k+3,\infty}(\Omega) \cap SB(\Omega,L,2k+2)}\hat{U}_g$, which is open, and $\Theta_{u,L,F} = \Gamma \cap B_R(0)$ where $B_{R}(0):=\{\theta \in \Rbb^{m}:\ |\theta| < R\}$ is a bounded open set for some $R>0$. Therefore $\Theta_{u,L,F}$ is also bounded and open. Moreover by Lemma \ref{lem:evolution_params}, for any $\theta \in \Theta_{u,L,F}$, there exists $\eta \in U_g$ for some $g$ and a differentiable $\hat{\theta}(t)$, such that $u_{\hat{\theta}(t)}(\cdot) = v_{\eta}(\cdot,t)$ for all $t$ with $\hat{\theta}(0)=\theta$. Therefore,
    \[
    \partial_t v_{\eta}(\cdot,0) = \partial_t u_{\hat{\theta}(t)}(\cdot)|_{t=0} = \nabla_{\theta} u_{\hat{\theta}(0)}(\cdot) \cdot \dot{\hat{\theta}}(0) = \nabla_{\theta} u_{\theta}(\cdot) \cdot \dot{\hat{\theta}}(0).
    \]
    Let $\alpha_{\theta} = \dot{\hat{\theta}}(0)$. Given that \eqref{eq:projection_proof} holds for $v_{\eta}$ and $u_{\theta}(\cdot) = v_{\eta}(\cdot,0)$, we know that
    \[
    \| \nabla_{\theta} u_{\theta} \cdot \alpha_{\theta} - F[u_{\theta}] \|_{W^{1,\infty}(\Omega)}
    = \|\partial_t v_{\eta}(\cdot,0) - F[v_{\eta}](\cdot,0)\|_{W^{k+1,\infty}(\Omega_{\delta})} < \epsilon,
    \]
    which justifies the first claim. 

    (ii) For any $g \in W^{2k+2,\infty}(\Omega) \cap SB(\Omega,L,2k+1)$, we know there exists $\eta_g \in U_g$, such that $v_{\eta_g}(\cdot,0) = g(\cdot)$. Given the definition of $\Theta_{u,F,L}$ above, there exists $\theta \in \Theta_{u,F,L}$ and a differentiable curve $\hat{\theta}(t)$ such that $\hat{\theta}(0)=\theta$ and $u_{\hat{\theta}(t)}(\cdot) = v_{\eta_g}(\cdot,t)$ for all $t$. Therefore, given that \eqref{eq:wg_bound} holds for $v_{\eta_g}$, we have
    \begin{equation*}
        \| u_{\theta} - g \|_{W^{1,\infty}(\Omega)} = \| u_{\hat{\theta}(0)} - g \|_{W^{1,\infty}(\Omega)} = \|  v_{\eta_g}(\cdot,0) - g(\cdot) \|_{W^{1,\infty}(\Omega)} < \epsilon,
    \end{equation*}
    which proves the second claim.
    \end{proof}

\section{Proof of Proposition \ref{prop:F_exists}}
\label{appx:B}
First we sight the following corollary to Theorem \ref{thm:theta_exist} which follows Lemma 3.3 in \cite{gaby2023neural}.
\begin{lemma}
\label{lem:bounded}
Suppose Assumption \ref{assump:regularity} is satisfied. For all $\varepsilon>0$ there exists $v: \bar{\Theta}_{u,F,L} \to \Rbb^m$ such that $v$ is bounded over $\bar{\Theta}_{u,F,L}$ and the value of $v$ at $\theta$, denoted by $v_{\theta}$, satisfies
\[
\|v_{\theta} \cdot \nabla_{\theta}\ut-F[\ut]\|_{H^2}\leq \varepsilon, \qquad \forall\, \theta \in \barTheta.
\]
\end{lemma}
Lemma \ref{lem:bounded} is an immediate result by combining Theorem \ref{thm:theta_exist} above and Lemma 3.3 of \cite{gaby2023neural}. Hence we omit the proof here.
With this lemma in hand, we can prove our main result of this section.
\begin{proof}[Proof of Proposition \ref{prop:F_exists}]
    In what follows, we will follow the proof of Proposition 3.4 in \cite{gaby2023neural} to extend that result from the $L^2$ to the $H^1$ norm.
    We first show that there exists a differentiable vector-valued function $V: \bar{\Theta}_{u,F,L} \to \Rbb^{d}$ such that 
\begin{equation}
\label{eq:V_bound}
    \|V(\theta) \cdot \nabla_{\theta}\ut-F[\ut]\|_{H^2} \leq \frac{\varepsilon}{2}
\end{equation}
for all $\theta \in \bar{\Theta}_{u,F,L}$. After which the claim about neural networks follows immediately from the proof of Proposition 3.4 in \cite{gaby2023neural}.
%
To this end, we choose $\bar{\varepsilon}_0 \in (0, \varepsilon/2)$ and $\bar{\varepsilon} \in (\bar{\varepsilon}_0, \varepsilon/2) $, then by Theroem \ref{thm:theta_exist} and Lemma \ref{lem:bounded} we know that there exists a neural network $\ut$, a bounded open set $\bar{\Theta}_{u,F,L} \subset \Rbb^{m}$, and $M_v>0$ such that there is a vector-valued function $\theta \mapsto v_{\theta}$, where for any $\theta \in \bar{\Theta}_{u,F,L}$, we have $|v_{\theta}|<M_v$ and
\[
\|v_{\theta} \cdot \nabla_{\theta}\ut-F[\ut]\|_{H^2} \leq \barvarepsilon.
\]
Note that $v_{\theta}$ is not necessarily differentiable with respect to $\theta$.
To obtain a differentiable vector field $V(\theta)$, for each $\theta \in \bar{\Theta}_{u,F,L}$, we define the function $\psi_{\theta}$ by
\begin{equation*}
    \psi_{\theta}(w) : = \| w \cdot \nabla_{\theta} \ut - F[\ut] \|_{H^2}^2 = w^{\top} G(\theta) w - 2 w^{\top} p(\theta) + q(\theta),
\end{equation*}
where 
\begin{equation}
\label{eq:def-G}
\begin{aligned}
    G(\theta)&=\sum_{i=1}^{d} \int_{\Omega}\nabla_{\theta} \partial_{x_i} u_{\theta}(x)\nabla_{\theta} \partial_{x_i} u_{\theta}(x)^{\top}dx+\int_{\Omega}\nabla_{\theta} u_{\theta}(x)\nabla_{\theta} u_{\theta}(x)^{\top}dx\\
    p(\theta)&=\sum_{i=1}^{d} \int_{\Omega}\nabla_{\theta} \partial_{x_i} u_{\theta}(x)F[\ut](x)dx+\int_{\Omega}\nabla_{\theta} u_{\theta}(x)\nabla_{\theta} F[\ut](x)dx\\
    q(\theta)&=\int_{\Omega} F[\ut](x)dx,
    \end{aligned}
\end{equation}
we are using the convention that $\nabla_{\theta}u_{\theta} \in \Rbb^d$ is a column vector.
%
Then we know 
\begin{equation}
\label{eq:psi_theta_star}
    \psi_{\theta}^{*} : = \psi_{\theta}(v_{\theta}) = \| v_{\theta} \cdot \nabla_{\theta} \ut - F[\ut] \|_{H^2}^2 \le \barvarepsilon^2.
\end{equation}
It is also clear that $G(\theta)$ is symmetric and positive semi-definite. Moreover, due to the compactness of $\bar{\Omega}$ and $\bar{\Theta}_{u,F,L}$, as well as that $\nabla_{\theta} u \in C(\barOmega \times \barTheta)$, we know there exists $\lambda_{G} > 0$ such that 
\begin{equation*}
    \| G(\theta) \|_2 \le \lambda_{G}
\end{equation*}
for all $\theta \in \bar{\Theta}_{u,F,L}$ with respect to the spectral norm.
Therefore, $\psi_{\theta}$ is a convex function and the Lipschitz constant of $\nabla \psi_{\theta}$ is uniformly upper bounded by $\lambda_{G}$ over $\bar{\Theta}_{u,F,L}$.
Now for any $w \in \Rbb^{m}$, $h>0$, and $K \in \Nbb$ (we reuse the letter $K$ as the iteration counter instead of the number of sampling points in this proof), we define
\begin{equation*}
    \Ocal_{\theta}^{K,h}(w) := w_{K}, \quad \mbox{where} \quad w_k = w_{k-1} - h \nabla \psi_{\theta}(w_{k-1}), \quad w_0 = w, \quad k = 1,\dots,K.
\end{equation*}
Namely, $\Ocal_{\theta}^{K,h}$ is the oracle of executing the gradient descent optimization scheme on $\psi_{\theta}$ with step size $h>0$ for $K$ iterations.

Now we note that $\psi_{\theta}$ is convex, differentiable, and $\nabla \psi_{\theta}$ is Lipschitz continuous with Lipschitz constant upper bounded by $\lambda_{G}$. With these in mind, we can directly use the process from \cite{gaby2023neural} to arrive at the inequality
\begin{equation}
\label{eq:psi_opt}
    \psi_{\theta}(\Ocal_{\theta}^{K,h}(0)) - \psi_{\theta}^{*}  \le  \Big( \frac{\varepsilon}{2} \Big)^2 - \bar{\varepsilon}^2.
\end{equation}
Here we have set $K$ to be some fixed number such that
\[
K \ge \frac{ M_v^2}{2h((\varepsilon/2)^2 - \bar{\varepsilon}^2)}.
\]

Notice that $\Ocal_{\theta}^{K,h}$ is a differentiable vector-valued function of $\theta$ because $K$ and $h$ are fixed.
Therefore, combining \eqref{eq:psi_theta_star} and \eqref{eq:psi_opt} yields
\begin{equation*}
    0\leq \psi_{\theta}(\Ocal_{\theta}^{K,h}(0)) = (\psi_{\theta}(\Ocal_{\theta}^{K,h}(0)) - \psi_{\theta}^{*}) + \psi_{\theta}^{*} \le (\varepsilon/2)^2 -\barvarepsilon^2 + \barvarepsilon^2 = (\varepsilon/2)^2.
\end{equation*}
As this inequality holds $\forall \theta \in \barTheta$, we set $V(\theta) = \Ocal_{\theta}^{K,h}(0)$ which is a differentiable function of $\theta$ satisfying \eqref{eq:V_bound}.
This completes the proof.
\end{proof}

\bibliographystyle{abbrv}
\bibliography{library}

\end{document}